\documentclass[12pt]{article}
\usepackage{amsmath, amsthm, amsbsy, amssymb, graphicx}
\usepackage{enumerate}
\usepackage{url} 

\newcommand{\PP}{\mathbb{P}}
\newcommand{\EE}{\mathbb{E}}
\newcommand{\RR}{\mathbb{R}}
\newcommand{\XX}{\mathbb{X}}
\newcommand{\ind}{\boldsymbol{1}}
\newcommand{\bZ}{\boldsymbol{Z}}

\newcommand{\bx}{\boldsymbol{x}}
\newcommand{\by}{\boldsymbol{y}}

\newcommand{\origin}{\boldsymbol{0}}
\newcommand{\bzero}{\boldsymbol{0}}
\newcommand{\dd}{\mathrm{d}}

\newcommand{\convp}{\stackrel{p}{\longrightarrow}}

\newtheorem{Theorem}{Theorem}[section]

\newtheorem{Definition}{Definition}[section]

\allowdisplaybreaks

\usepackage{color}
\newcommand{\sid}[1]{{\color{black} #1}}
\newcommand{\twang}[1]{{\color{black} #1}}

\begin{document}
\bibliographystyle{plain}

\def\spacingset#1{\renewcommand{\baselinestretch}%
{#1}\small\normalsize} \spacingset{1}


{
  \title{\bf 2RV+HRV and Testing for Strong VS Full Dependence}
  \author{Tiandong Wang\thanks{
    The first author gratefully acknowledges \textit{National Natural Science Foundation of China Grant 12301660 and Science and Technology Commission of Shanghai Municipality Grant 23JC1400700}.}\hspace{.2cm}\\
    Shanghai Center for Mathematical Sciences, Fudan University\\
    and \\
    Sidney I. Resnick \\
    School of Operations Research and Information Engineering, Cornell University}
  \maketitle
} 

{
  \bigskip
  \bigskip
  \bigskip
  \begin{center}
    {\LARGE\bf 2RV+HRV and Testing for Strong VS Full Dependence}
\end{center}
  \medskip
} 

\bigskip
\begin{abstract}
Preferential attachment models of network growth are
  bivariate heavy tailed models for in- and out-degree with limit
  measures which either concentrate on a ray of positive slope from
  the origin or on all of the positive quadrant depending on whether
  the model includes reciprocity or not.  Concentration on the ray is
  called full dependence. If there were a reliable way to distinguish
  full dependence from not-full, we would have guidance about which
  model to choose. This motivates investigating tests that distinguish
  between (i) full dependence; (ii) strong dependence (support of the limit measure
  is a proper subcone of the positive quadrant); (iii)
  weak dependence (limit measure concentrates on positive quadrant). We give two test statistics,
  analyze their asymptotically normal behavior under full and not-full
  dependence, and discuss applicability using bootstrap methods applied to simulated and real data.
\end{abstract}

\noindent%
\vfill

\spacingset{1.9} 

\section{Introduction}
In multivariate heavy tail estimation, the support of the limit
measure provides information on the dependence structure of the random
vector with the heavy tail distribution
(\cite{lehtomaa:resnick:2020}).
However, even in 
favorable circumstances in $\RR_+^2$, the positive quadrant in two dimensions,
scatter or diamond plots may have trouble distinguishing between 
\begin{itemize}
  \item {\it Full dependence\/} where the limit measure concentrates
    on a ray of the form $\{(x,y)\in \RR_+^2: y/x=m>0\}$;
    \item {\it Strong dependence\/} where \twang{the support of the limit measure
      is} a proper subcone of $\RR_+^2$ of the form
      $\{(x,y)\in \RR_+^2: y/x \in [m_l,m_u]\subsetneq [0,1]\}$;
    \item {\it Weak dependence\/} where the support of the limit measure is all of
       $\RR_+^2$; and
     \item {\it Asymptotic independence\/} where the limit measure
      concentrates on the axes $\RR_+\times \{0\} \cup \{0\}\times
      \RR_+$.
       \end{itemize}

Estimation and visualization techniques
that attempt to accurately
distinguish these  cases encounter complications, the most glaring of
which is the  requirement that data  be thresholded 
according to the distance from the origin. Plots can look
rather different depending on the choice of threshold. This is
illustrated by diamond plots \cite{das:resnick:2017,lehtomaa:resnick:2020}
in Figure \ref{fig:1} of Exxon (XOM) returns vs returns from Chevron (CVX) from
January 04, 2016 to December 30, 2022. The data $\{(x_i,y_i); 1 \leq i \leq 1761\}$ is mapped
to the $L_1$-unit sphere via $(x,y)\mapsto (x,y)/(\vert x\vert+\vert y
\vert)$ and then subsetted by retaining only the points with $k$
largest values of $(\vert x\vert+\vert y
\vert)$ where $k=100$ (left) or $k=500$ (right). Unsurprisingly, the
two plots give different impressions of where the limit measure
concentrates. 

\begin{figure}[h]
  \center
\includegraphics[width = \textwidth]{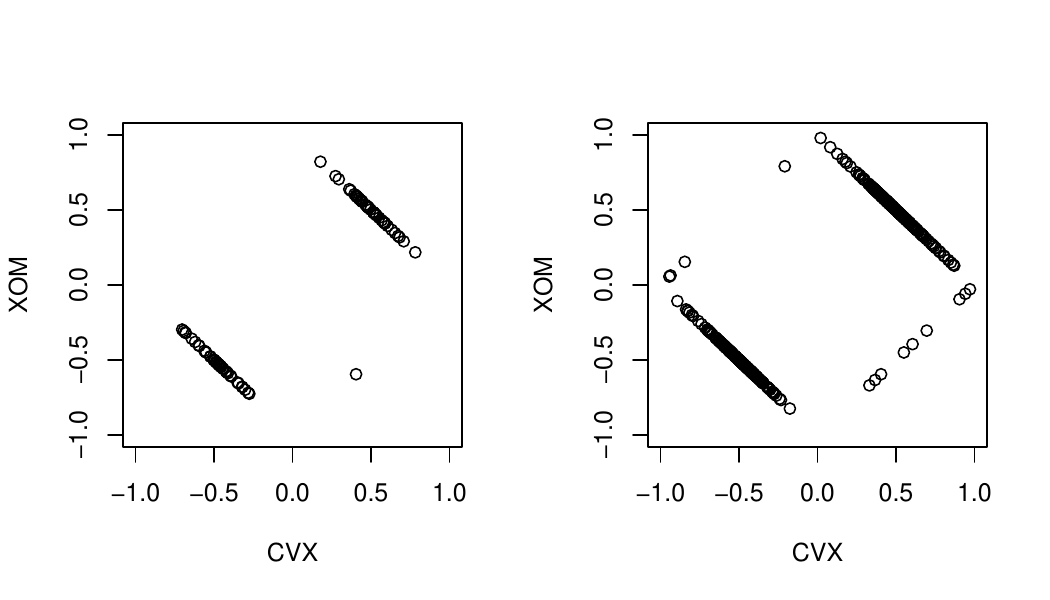}
\caption{Left: Diamond plot using the 100 points furthest from the
  origin. Right: Diamond plot using 500 most remote points.}
\label{fig:1}
\end{figure}
\medskip

An automatic
threshold selection technique is advocated in computer and
network science 
(\cite{clauset:shalizi:newman:2009,virkar:clauset:2014}) and implemented
in \cite{csardi:nepusz:2006} or \cite{gillespie:2015}. This technique is consistent
(\cite{bhattacharya:chen:vanderhofstad:zwart:2020, drees:janssen:resnick:wang:2020})
and can increase
ones' comfort level with threshold selection. However,  this method
offers no guarantee of best choice of threshold 
and has the additional drawback of preventing tail estimators such as Hill from
being asymptotically normal
(\cite{drees:janssen:resnick:wang:2020}). It would be desirable if there
were test statistics to guide us in choice of model from the bulleted
list above.

One reason for thinking about such problems was our interest in fitting
preferential attachment (PA) models of directed social networks
to data consisting of in- and out-degree of each node. The
classical PA model of directed edge growth
(\cite{krapivsky:redner:2001,bollobas:borgs:chayes:riordan:2003}) when
standardized to equal tail indices for each component gives a heavy
tail model with limit
measure concentrating on all of $\RR_+^2$
(\cite{resnick:samorodnitsky:towsley:davis:willis:wan:2016}). However,
for reasonable ranges of model parameters, these models 
do not correctly predict empirically observed
values for the reciprocity coefficient (\cite{kunegis:2021}); this is 
discussed in
\cite{wang:resnick:2022b, wang:resnick:2022a,
  cirkovic:wang:resnick:2022c,wang:resnick:2023hetero}. Adding the
reciprocity feature to the theoretical model means predicted values of
the reciprocity coefficient can match empirical values. However,
unlike the classical model, this heavy tailed model with reciprocity  has a limit
measure that concentrates on a ray. If there were a statistical test for
full dependence, it would provide guidance on whether one needs to add
the reciprocity feature to the model to obtain a satisfactory
statistical fit for social network data.

Of course network data or financial returns are not the same as iid observations but this
paper starts with the simple case and assumes all observations come
from a heavy tailed iid model by repeated sampling.

We give two test statistics which help distinguish full vs not-full
asymptotic dependence and show the statistics are asymptotically
normal but with different asymptotic variances, depending on the case.
A somewhat novel aspect of our approach is that hidden regular
variation (HRV) arises naturally and is employed in our proofs. The
reason HRV is relevant is that to get asymptotic normality with a
constant centering for estimators of heavy tailed data requires not
only the regular variation assumption for the underlying distribution
but also second order regular variation (2RV) which controls
deviations between a finite sample mean and an asymptotic mean; this
is discussed at 
length in
\cite{dehaan:ferreira:2006,resnickbook:2007,dehaan:resnick:1996,dehaan:1996,
  peng:1998}. In the context of two dimensional data, it is natural to
expect we need a two-dimensional second order regular variation
assumption (\cite{dehaan:resnick:1993,dehaan:deronde:1998,
  resnick:2002,das:kratz:2020}) and this coupled with multivariate
regular variation with a limit measure concentrating on a proper
subset of the state space lead naturally to hidden regular
variation. This is explained further in Section \ref{sec:rv}.

We propose test statistics that offer guidance on appropriateness
of the different cases and give conditions under which the statistics
are asymptotically normal, paying attention to the centering and
asymptotic variance. The proposed hypothesis testing framework is 
\sid{discussed further in Section \ref{sec:boot}.}

\section{Multivariate, Hidden and Second Order Regular Variation}\label{sec:rv}
Here is a review of notation and concepts necessary for formulating
and proving the results. We particularize the concept of regular
variation of measures on a complete, separable metric space $\XX$
for the
case of $\XX=\RR^2_+$ where visualization is easiest
(\cite{lindskog:resnick:roy:2014,hult:lindskog:2006a,das:mitra:resnick:2013,kulik:soulier:2020,basrak:planinic:2019}).


Suppose $\{(X_i,Y_i); 1 \leq i \leq n\}$ are iid random elements of
$\RR_+^2 $ and based on evidence provided by observing these vectors,
we need to analyze the 
asymptotic dependence structure of the components. We do this in the
context of regular variation of measures.

\subsection{Multivariate regular variation.}\label{sub:mrv}
We
begin with the concept of
$\mathbb{M}$-convergence.

\begin{Definition}\label{def:Mconv} For a closed subcone $\mathbb{C}$ of $\mathbb{X}$,
let $\mathbb{M}(\mathbb{X}\setminus \mathbb{C})$ be the set of Borel
measures on $\mathbb{X}\setminus \mathbb{C}$ which are finite on
sets bounded away from $\mathbb{C}$, and
$\mathcal{C}(\mathbb{X}\setminus \mathbb{C})$ be the set of
continuous, bounded, non-negative functions on $\mathbb{X}\setminus
\mathbb{C}$ whose supports are bounded away from  $\mathbb{C}$. Then for $\mu_n,\mu \in \mathbb{M}(\mathbb{X}\setminus
\mathbb{C})$, we say $\mu_n \to \mu$ in
$\mathbb{M}(\mathbb{X}\setminus \mathbb{C})$, if $\int f\dd
\mu_n\to\int f\dd \mu$ for all $f\in \mathcal{C}(\mathbb{X}\setminus
\mathbb{C})$. 
\end{Definition} 

Without loss of generality (\cite{lindskog:resnick:roy:2014}), we can
 take functions in  $\mathcal{C}(\mathbb{X}\setminus
\mathbb{C})$ to be uniformly continuous.
The modulus
of continuity of a uniformly continuous function $f:\RR_+^p \mapsto
\RR_+$ is
\begin{equation}\label{e:defModCon}
\Delta_f(\delta)=\sup\{ \vert f(\bx)-f(\by)\vert :
d(\bx,\by)<\delta\}\end{equation}
 where $d(\cdot,\cdot)$ is an appropriate metric
on the domain of $f$. Uniform continuity means $\lim_{\delta \to 0}
\Delta_f(\delta)=0.$

Denote by $RV_c$, the  class of
regularly varying functions $f:\RR_+\mapsto \RR_+$ with index
$c\in\RR$ and write $f\in RV_c$.
The formal definition of multivariate regular variation  (MRV) of
distributions for the 
classical case  $\XX=\RR_+^2$ and
$\mathbb{C}=\{\bzero\}$ is next.

\begin{Definition}\label{def:MRV}
The distribution $\PP [(X_1,Y_1)\in\cdot \,]$ of a
 random vector $(X_1,Y_1)$ on $\mathbb{R}_+^2$,
  is (standard) regularly varying on $\mathbb{R}_+^2\setminus
  \{\origin\}$ with index $\alpha>0$  if
  there exists some {regularly varying} scaling function $b(t)\in \text{RV}_{1/\alpha}$ and a
  limit measure $\eta (\cdot)\in \mathbb{M}(\mathbb{R}_+^2\setminus
  \{\origin\})$ such that 
as $t\to\infty$,
\begin{equation}\label{eq:def_mrv}
t\PP [(X_1,Y_1)/b(t)\in\cdot \,]\rightarrow \eta(\cdot),\qquad\text{in }\mathbb{M}(\mathbb{R}_+^2\setminus \{\origin\}).
\end{equation}
It is  convenient to write  $\PP [(X_1,Y_1)\in\cdot \,]\in  
  \text{MRV}(\alpha, b(t), \eta, \mathbb{R}_+^2\setminus \{\origin\})$.
\end{Definition}

\subsubsection{Cartesian to polar and back.}\label{subsub:cartPolar}
When analyzing the asymptotic dependence between components of a
bivariate random vector
$\bZ$ satisfying \eqref{eq:def_mrv},  it is often informative  
to make a polar coordinate transform and
consider the transformed points located on the $L_1$-unit sphere
\begin{align}
\label{eq:map_L1}
(x,y)\mapsto\left(\frac{x}{x+y},\frac{y}{x+y}\right),
\end{align}
after thresholding the data according to the
$L_1$ norm. The plot of such data is the (positive-quadrant) diamond
plot and Figure \ref{fig:1}  is the 4-quadrant version.
\sid{In $\RR_+^2$, the convenient version of the}
$L_1$-polar coordinate transformation is
$T:\RR_+^2 \setminus \{\bzero\} \mapsto (\RR_+\setminus \{0\}) \times
[0,1],$ defined by 
$$T(x,y) = \bigl(x+y, \sid{x/(x+y)\bigr) =(r,\theta)}.$$ The inverse
transformation from polar to Cartesian coordinates is
$(r,\theta)\mapsto (r\theta, r(1-\theta)).$ The map $T$
disintegrates  $\eta(\cdot)$ into the product measure
$$\eta \circ T^{-1}(\cdot) =\nu_\alpha \times S(\cdot),$$
where $S(\cdot)$ can be taken to be a probability measure on $[0,1]$
called the angular measure, and $\nu_\alpha (\cdot)$ is the Pareto
measure with $\nu_\alpha (x,\infty)=x^{-\alpha},\,x>0$.

\subsection{Hidden regular variation.}\label{sub:hrv}
Denote by $\mathbb{C}_{a,b}$ the subcone of $\RR_+^2 $ such that
$$\mathbb{C}_{a,b}=\{(x,y)\in \RR_+^2: \theta=x/(x+y) \in
[a,b]\},\quad 0\leq a \leq b\leq 1.$$
When the limit measure of regular variation  $\eta(\cdot)$ 
concentrates on a proper subcone $\mathbb{C}_{a,b}\subset \XX =\RR_+^2$ of the full state space,
we may improve estimates of probabilities in the complement of the subcone,
if there is a
second {\it hidden\/} regular variation regime after removing the
subcone. 

\begin{Definition}
  The random vector $\bZ$  in $\RR^2_+$ has a distribution
  that is regularly varying on $\RR^2_+ \setminus \{\origin\}$ and has hidden regular variation on $\RR^2_+
\setminus\mathbb{C}_{a,b}$ if there exist $0 <\alpha\le \alpha_0$, scaling
functions $b(t) \in RV_{1/\alpha}$ and $b_0(t) \in RV_{1/\alpha_0}$
with $b(t)/b_0(t) \to\infty$ and limit measures $\eta$ concentrating
on $\mathbb{C}_{a,b}$ and another limit measure $\eta_0$, such
that 
\begin{align}\label{eq:def_hrv}
\PP(\bZ\in\cdot)\in 
\text{MRV}(\alpha, b(t), \eta, \mathbb{R}_+^2\setminus \{\origin\})
\cap \text{MRV}(\alpha_0, b_0(t), \eta_0, \mathbb{R}_+^2\setminus \mathbb{C}_{a,b}).
\end{align}
\end{Definition}

It is sometimes useful to characterize HRV is through the
\emph{generalized polar coordinate transformation} for
$\mathbb{R}_+^2\setminus \mathbb{C}_{a,b}$, assuming use of an associated metric
$d(\cdot,\cdot)$ satisfying $d(cx,cy) = cd(x,y)$ for scalars
$c>0$. The metric $d(\cdot,\cdot)$ that we use in practice is \sid{a
  scaled} 
$L_1$-metric. 
When using generalized polar coordinates with respect to the forbidden
zone $\mathbb{C}_{a,b}$, we define $\aleph_{\mathbb{C}_{a,b}} :=\{\bx
\in\mathbb{R}_+^2\setminus \mathbb{C}_{a,b} : d(\bx,\mathbb{C}_{a,b}) = 1\}$, the
locus of points at distance 1 from 
$\mathbb{C}_{a,b}$. 
Then the generalized polar coordinates are specified through the transformation,
$\text{GPOLAR}: \RR_+^2\setminus \mathbb{C}_{a,b}\mapsto (0,\infty)\times \aleph_{\mathbb{C}_{a,b}}$ with
\begin{align}
\label{eq:GPOLAR}
\text{GPOLAR}(\bx) = \left(d(\bx,\mathbb{C}_{a,b}),\frac{\bx}{d(\bx,\mathbb{C}_{a,b})}\right).
\end{align}
For a probability measure
$S_0(\cdot)$  on $\aleph_{\mathbb{C}_{a,b}}$, the generalized polar
coordinates allow re-writing the second part of \eqref{eq:def_hrv} as 
\[
t\PP\left[\left(\frac{d(\bZ, \mathbb{C}_{a,b})}{b_0(t)}, \frac{\bZ}{d(\bZ, \mathbb{C}_{a,b})}\right)\in\cdot\right]
\to (\nu_{\alpha_0}\times S_0)(\cdot)
\]
in $\mathbb{M}\bigl((\RR_+\setminus\{0\})\times
\aleph_{\mathbb{C}_{a,b}}\bigr)$. In particular $\PP [d(\bZ, \mathbb{C}_{a,b})>x] \in
RV_{-\alpha_0}$ is a lighter tail than $\PP [\Vert \bZ \Vert >x]\in
RV_{-\alpha}$. 
See \cite{das:mitra:resnick:2013} and \cite{lindskog:resnick:roy:2014} for details.

\subsection{Second order regular variation.}\label{sub:2RV}
In one dimension, second order regular variation (2RV) controls deviation of
finite sample means from asymptotic means and allows a useful
asymptotic normality for estimators such as the Hill estimator.
Our test statistics are
derived from two dimensional tail empirical measures and it is
reasonable to expect, therefore, that asymptotic normality requires
two dimensional second order
regular variation conditions.

\subsubsection{The second order condition.}\label{subsub:2rv}
There are several ways to state this condition which strengthens
multivariate regular variation. The first uses
$\mathbb{M}$-convergence. We need
a  function $A \in RV_{-\rho}$, $\rho>0$, and a
signed measure $\chi(\cdot) $  which is
not identically 0 and is 
the difference of two measures in 
$\mathbb{M} \bigl( (\RR_+\setminus \{0\}) \times
[0,1]\bigr) $,
such that in
  $\mathbb{M} \bigl( (\RR _+\setminus  \{0\}) \times [0,1] \bigr)$,
  \begin{align}\label{e:2RVmeas}
  \frac{1}{A(t)} \Biggl( t\PP \Bigl( (R/b(t), \Theta) \in \cdot \Bigr) - \nu_\alpha \times
  S (\cdot) \Biggr) \to \chi (\cdot),
  \end{align}
  meaning that evaluation of the signed measure on the left at a
  function $f\in \mathcal{C} (\bigl( (\RR _+\setminus  \{0\}) \times
  [0,1] \bigr)$ converges to the evaluation $\chi (f)$; or in symbols
\begin{equation}\label{e:withf}
  \frac{1}{A(t)} \Biggl( t\EE  f\bigl( R/b(t), \Theta\bigr)  
  -\iint_{\RR_+\setminus \{0\}\times [0,1] }f(r,\theta) \nu_\alpha (dr)  S (d\theta) \Biggr) \to \chi (f).
\end{equation}

The second way
  to phrase condition \eqref{e:2RVmeas} which looks more like convergence of
  distribution functions is
\begin{equation}\label{e:2rv}
  \frac{1}{A(t) } \Biggl( t\PP \Bigl(\frac{R}{b(t)}>r,\Theta \leq
    \theta\Bigr)
    -r^{-\alpha}S[0,\theta]\Biggr) \to \chi\bigl( (r,\infty) \times
    [0,\theta] \bigr )
  \end{equation}
\sid{locally uniform in $r\in (0,\infty)$ for each $\theta \in
  [0,1]$ where the limit is specified before \eqref{e:2RVmeas}.}

If $f_1(r) \in \mathbb{M}(\RR_+\setminus \{0\})$, set
$f(r,\theta):=f_1(r)\theta \in \mathbb{M} \bigl( (\RR _+\setminus
\{0\}) \times [0,1] \bigr)$ and inserting this into \eqref{e:withf}
gives
\begin{equation}\label{e:expVer}
  \frac{1}{A(t)} \Biggl( t\EE \Theta f_1\bigl( R/b(t) \bigr)  
  -\int_{[0,1]}\theta S(d\theta) \nu_\alpha (f_1)\Biggr) \to
  \int_{(0,\infty)\times [0,1]}\theta f_1(r)\chi (dr,d\theta).
\end{equation}
or in convergence of signed measures formulation,
\begin{equation}\label{e:expVerMeas}
  \frac{1}{A(t)} \Biggl( t\EE \Theta \epsilon_{ R/b(t) }(\cdot)
  -\int_{[0,1]}\theta S(d\theta) \nu_\alpha (\cdot)\Biggr) \to
  \iint_{(\cdot)\times [0,1]} \theta \chi (dr,d\theta).
\end{equation}
\sid{Note that \eqref{e:2RVmeas} and \eqref{e:2rv} are formulated so
  they can be marginalized and therefore the regularly varying distribution of
  $R$ is 2RV in one dimension. Also, the second order condition allows 
\eqref{e:expVerMeas} and \eqref{e:expVer}, so controls the expectation
of $\Theta$ on the set where $R$ is large. If we set
$$v(t)=E\Theta_1 \ind_{[R_1>t]},\quad \mu_S=\int_{[0,1]}\theta S(d\theta),$$
then \eqref{e:expVerMeas} gives as $t\to \infty$,
\begin{equation}\label{e:needE}
\frac{  tv(b(t)x)-\mu_S x^{-\alpha}}{A(t)} \to h(x):=
\iint_{((x,\infty))\times [0,1]} \theta \chi (dr,d\theta).
\end{equation}
which leads to  the more traditional form of the 2RV condition for
$v(t)$, namely
\begin{equation}\label{e:needE2}
\lim_{s\to\infty}  \frac{ \frac{v(sx)}{v(s)} -x^{-\alpha}    }{A\circ
  b^\leftarrow (s)}
=h(x)/\mu_S,
\end{equation} where  $A\circ
  b^\leftarrow  \in RV_{-\rho \alpha}$ and the limit function $h(x)$
  must be of the form (\cite{dehaan:stadtmueller:1996,
    peng:1998, dehaan:ferreira:2006}),
  $$h(x)=cx^{-\alpha} \Bigl( \frac{1-x^{-\rho \alpha  }}{\rho \alpha}
  \Bigr), \quad x>0,\, c\neq 0.$$
}

\subsubsection{2RV and HRV}\label{subsub:2+H}
We discuss why the second order condition \eqref{e:2RVmeas} together
with the 
 assumption $S([a,b]) = 1$ for 
 $[a,b]\subsetneq [0,1]$ implies HRV. The essentials of the argument
 in the context of asymptotic independence are in
  (\cite{dehaan:deronde:1998, resnick:2002}).
  
  \begin{Theorem}
    Assume the 2RV condition     \eqref{e:2RVmeas} or \eqref{e:2rv} 
 hold and $S([a,b]) = 1$ for  $[a,b]\subsetneq [0,1]$.
Set $U(t) =t/A(t) \in RV_{1+\rho}$, so that
    $U^\leftarrow (t) \in RV_{1/(1+\rho)} $ and therefore
    \begin{equation}\label{e:defb0}
      b_0(t):=b\circ U^\leftarrow (t) \in  RV_{1/(\alpha(1+\rho))},
      \,\rho>0.
    \end{equation}
  Then provided $\chi(\cdot)$ is not identically 0 on
  $(0,\infty)\times([0,1]\setminus [a,b])$,
\begin{align}
\label{eq:HRV}
\PP\left[\left(R,\Theta\right)\in\cdot\right]\in\text{MRV}(\alpha(1+\rho),
  b_0(t), \chi (\cdot), (\RR_+\setminus \{0\})\times([0,1]\setminus [a,b])).
\end{align}
\end{Theorem}

\begin{proof}
  For $r>0$, $I\subset [0,1]$ such that
  $\chi(\partial ((r,\infty)\times  I))=0$, 
%
\begin{align*}
\chi((r,\infty)\times I)=&\lim_{t\to \infty} 
                           \frac{t
                           \PP[R/b(t)>r,\Theta \in I  ]- r^{-\alpha}S(I)}{ A(t) }\\
\intertext{and if $S(I)=0$, this is }
=&\lim_{t\to \infty}
\frac{t\PP[R/b(t)>r,\Theta \in I  ]}{ A(t) }
\end{align*}
Set $U(t):= t/A(t) $, $b_0:= b\circ U^\leftarrow$
and after a change of variable, the proof of
\eqref{eq:HRV}  is
complete. 
\end{proof}

\section{Testing the existence of strong dependence}
For strong convergence,
we assume that $0\leq a\leq b\leq1$ fixed, with
$[a,b]\subsetneq [0,1]$ and $S([a,b])=1.$
The condition $\theta =x/(x+y) \in [a,b]$
translates to
$$(x,y)\in \{(u,v)\in \RR_+^2:  v/u \in [b^{-1}-1 , a^{-1}
-1]\}=:\mathbb{C}_{a,b}.$$
So the closed cone $\mathbb{C}_{a,b}$ is the set of first quadrant
  points between the two rays $y=m_u x$ and $y=m_l x$, $x>0$, where
  the slopes are $m_u=a^{-1}-1$, $m_l=b^{-1}-1$
  and since $a\leq b$, we have $m_u \geq m_l$.
Define the scaled distance from $(x,y)\in
\mathbb{R}_+^2$ to $\mathbb{C}_{a,b}$ as 
\begin{equation}
d^*((x,y),\mathbb{C}_{a,b}) := \max\left\{(b^{-1}-1)x-y, y-(a^{-1}-1)x,
  0\right\}. \label{e:d}
\end{equation}
Note
\begin{enumerate}
\item when $(x,y)$ is above cone $\mathbb{C}_{a,b}$ so that $y/x >m_u$ and
thus $y > (a^{-1}-1)x$, $d^*((x,y),\mathbb{C}_{a,b}) = y-(a^{-1}-1)x;$
\item when $(x,y)$ is below the cone $\mathbb{C}_{a,b}$ so that $y/x <m_l$
  and $y<(b^{-1}-1)x$, $d^*((x,y),\mathbb{C}_{a,b}) =  (b^{-1}-1)x -y;$
\item when $(x,y  ) \in \mathbb{C}_{a,b}$, $d^*((x,y),\mathbb{C}_{a,b}) =0.$
\item when $\mathbb{C}_{a,b}=\{(x,y): y=(\theta_0^{-1} -1)x, x>0\}$
  because $S\{\theta_0\}=1$,  then $d^*((x,y),\mathbb{C}_{a,b})=
  \vert (\theta_0^{-1}-1)x-y  \vert .$
\end{enumerate}

Using generalized polar coordinates, the HRV assumption on
$\RR_+^2\setminus\mathbb{C}_{a,b}$ reads 
\[
  t\PP\left(\left(
      \frac{d^*\bigl((X,Y),\mathbb{C}_{a,b}\bigr)}{b_0(t)},
      \frac{(X,Y)}{d^*((X,Y),\mathbb{C}_{a,b})} \right)\in
  \cdot\right)\longrightarrow \nu_{\alpha_0} \times S_0(\cdot) 
\]
in $\mathbb{M}( (\RR_+\setminus \{0\}) \times \aleph_{\mathbb{C}_{a,b}})$ and in
particular
$P[d^*\bigl((X,Y),\mathbb{C}_{a,b} \bigr)>x] \in RV_{-\alpha_0}$ \sid{and
  assuming 2RV 
  from the previous section, $\alpha_0=\alpha(1+\rho)$.}

Let $\{(X_i,Y_i): i\ge 1\}$ be iid copies of $(X,Y)$, and set $R_i:= X_i+Y_i$.
We also define $(X_i^*, Y_i^*)$ to be the pair of random variables such that $X_i^*+Y_i^*$ is the $i$-th largest order statistic of $\{R_i: 1\le i\le n\}$, i.e. $R_{(i)}$.
Consider the following hypotheses: for fixed and known $0<a\le b<1$,
\begin{align}
H^{(1)}_0: \, S([a,b]) = 1,\qquad H^{(1)}_a:\, S([a,b]) < 1.
\label{eq:test_strong}
\end{align}

In Theorem~\ref{thm:test_strong}, we propose a test statistic for the test in \eqref{eq:test_strong}. 

\begin{Theorem}\label{thm:test_strong}
\sid{Assume the 2RV condition} \eqref{e:2rv} holds, 
$\alpha_0\equiv \alpha(1+\rho)>1$ and $b_0(t)$ is defined in \eqref{e:defb0}.
Define the statistic
\begin{align}
D^*_n:=&\frac{1}{{k_n}}\sum_{i=1}^{k_n}
         \left(1+\frac{d^*\bigl((X^*_i,Y^*_i),\mathbb{C}_{a,b}\bigr)}{R_{(k_n)}}\right)\log\frac{R_{(i)}}{R_{(k_n)}}\nonumber \\
=&H_{k,n} +\frac{1}{{k_n}}\sum_{i=1}^{k_n}
   \left(\frac{d^*\bigl((X^*_i,Y^*_i),\mathbb{C}_{a,b}\bigr)}{R_{(k_n)}}\right)\log\frac{R_{(i)}}{R_{(k_n)}}\label{e:defD*}
\end{align}
where $H_{k_n,n}$ is the Hill estimator of $1/\alpha$ applied to $\{R_i,
1\leq i\leq n\}$ based on $k_n$ upper order statistics, and
 $\{k_n\}$ is an intermediate sequence \sid{(i.e. $k_n\to \infty$, $n/k_n
 \to \infty, $ $n\to\infty$)} satisfying 
\begin{align}
\label{eq:cond1}
\sqrt{k_n} \frac{b_0(n/k_n)}{b(n/k_n)}\to 0,\qquad n\to\infty.
\end{align}
Under $H^{(1)}_0$ as given in \eqref{eq:test_strong}, we have
\begin{equation}\label{e:AN}
\sqrt{k_n}(D^*_n - 1/\alpha)\Rightarrow \frac{1}{\alpha}N(0,1).
\end{equation}
\end{Theorem}

\twang{The proof of Theorem~\ref{thm:test_strong} is  in
  Section 1 of the supplement;  here we give some remarks:}
\begin{enumerate}
  \item Of course, $D_n^*$ depends on $a,b$ but this dependence is
    suppressed in the notation. A consistent estimator of $a,b$ is
    suggested in Section \ref{subsub:estimate} but for now we assume
    $a,b$ are fixed and known.
\item Under $H^{(1)}_0$, for $(X_i, Y_i)$ such that $R_i$ is large, the distance from $(X_i,Y_i)$ to $\mathbb{C}_{a,b}$
should be small with high probability and therefore $D_n^*$
  should be close to the Hill estimator which is asymptotically
  normal.
  The proof of Theorem \ref{thm:test_strong} shows that when
  $S[a,b]=1$, 
  \begin{equation}\label{e:small}
\sqrt{k_n } (D_n^* -H_{k_n,n})\Rightarrow 0,\quad
(n\to\infty).\end{equation} 
\item 
The condition $\alpha_0=\alpha (1+\rho)>1$ is mild as it is 
  rare in practice for tails to be so heavy that $\alpha<1$.
\item
The proof is based on asymptotic normality of the tail
  empirical measure. For treatments explaining the need for the second
  order condition, see \cite[Section 9.1]{resnickbook:2007} or 
  \cite{dehaan:ferreira:2006} and for  background
  \cite{straf:1972,csorgo:mason:1985,dehaan:resnick:1993,dehaan:resnick:1998,
    einmahl:1987,csorgo:haeusler:mason:1991a,
   csorgo:haeusler:mason:1991b,
   mason:turova:1994}.
\item  Theorem~\ref{thm:test_strong} suggests that for fixed $a,b$, we  reject
  $H^{(1)}_0$ in \eqref{eq:test_strong} if 
  $\vert D_n^*-1/\alpha\vert>1.96/(\alpha/\sqrt{k_n})$.
 If we choose too wide an interval $[a,b]\subsetneq [0,1]$, then the test statistic $D_n^*$
becomes closer to $H_{k,n}$ as more data points are included in
$\mathbb{C}_{a,b}$.  Failure to reject for the fixed interval means
also that one fails to reject for any bigger interval. So
using only $D_n^*$, we cannot distinguish whether the  
support of $S(\cdot)$ is in $[a,b]$ or a subset of $[a,b]$ and, in particular,
if we fail to reject $H_0^{(1)}$, it could be the support is
$\{\theta_0\}$ for some 
$\theta_0\in [a,b]$. Therefore, in the next section, we give another test
statistic that helps decide whether $(X,Y)$ is  
asymptotically fully or strongly dependent.

\item If $[a,b]=[0,1]$, then $\mathbb{C}_{a,b}=\RR_+^2$ and
$d^*\bigl((X^*_i,Y^*_i),\mathbb{C}_{a,b}\bigr)  =0$ so $D_n^*=H_{k_n,n} $
and \eqref{e:AN} still holds \sid{without any restriction on $\{k_n\}$
  beyond it being an intermediate sequence.}

\end{enumerate}


\section{Full vs strong dependence}
Consider the following hypothesis test: 
\begin{align}\label{eq:test_fullstrong}
  H^{(2)}_0: S(\{\theta_0\}) = 1\qquad H^{(2)}_a: S([a,b])=1.
\end{align} 
where $\theta_0\in [a,b]$, and to capitalize on hidden regular variation
resulting from 2RV, we need the assumption that $[a,b]\subsetneq
[0,1]$ is a proper subset of $[0,1].$
Since $\theta_0\in [a,b]$ and $D_n^*$ given in Theorem~\ref{thm:test_strong} is unable to 
distinguish between the two hypotheses in \eqref{eq:test_fullstrong}, we now propose another test statistic.
Let $\Theta_i^*$ be the concomitant of $R_{(i)}$, and define
\begin{align}
\label{eq:defTn}
T_n := \frac{\sum_{i=1}^{k_n}\Theta_i^*\log\frac{R_{(i)}}{R_{(k_n)}}}{\sum_{i=1}^{k_n}\Theta_i^*}.
\end{align}

The next results \sid{recommend} we distinguish between strong and
full dependence by assessing the asymptotic variance of
$T_n$. \sid{The methodology is} 
discussed more fully in the next Section \ref{sec:boot}. Under $H^{(2)}_0$
the asymptotic variance of $T_n$ is $1/\alpha^2$ but under $H^{(2)}_a$ the
asymptotic variance is strictly greater than $1/\alpha^2$.

\subsection{Full dependence}\label{sub:full}
We begin by discussing a limit theorem that can aid in 
distinguishing between full and strong dependence.
This theorem is posed under the assumption 
$H^{(2)}_0$ in \eqref{eq:test_fullstrong} that full dependence holds with the limit
angular measure concentrating at a point $\theta_0 \in (0,1)$.
The proof machinery is similar to that later in Theorem~\ref{thm:strong_full_HA}, with details included in the supplement.

\begin{Theorem}\label{thm:strong_full}
Assume $H^{(2)}_0$ holds and the angular measure $S(\cdot)
=\epsilon_{\theta_0}(\cdot),\,\theta_0 \in (0,1)$.
Suppose the 2RV condition in \eqref{e:2RVmeas} holds with \sid{ $A(t) \in
RV_{-\rho}$, $\rho>0$.}  Define \sid{$b_0(t) $ as in \eqref{e:defb0}}
so $b_0(t) \in RV_{1/(\alpha(1+\rho))}$ and $\alpha_0=\alpha(1+\rho)$.
 Let $\{k_n\}$ be an intermediate sequence satisfying \eqref{eq:cond1}.Then for $W(\cdot)$ a standard Brownian Motion  we have 
\begin{align}
\sqrt{k_n}\left(T_n - \frac{1}{\alpha}\right) &\Rightarrow
  \frac{1}{\alpha}\left(\int_0^1 W(s)\frac{\dd
  s}{s}-W(1)\right)
  \sid{\stackrel{d}{=} \frac{1}{\alpha}W(1)\sim
  N(0,1/\alpha^2).}\label{e:fullvar}
\end{align}
\end{Theorem}

\subsection{Strong dependence}\label{sub:strong}
Theorem~\ref{thm:strong_full_HA} suggests 
identifying strong dependence in $H^{(2)}_a$ if the asymptotic variance of
$T_n$ is bigger than $1/\alpha^2$.

\begin{Theorem}\label{thm:strong_full_HA}
Consider the hypothesis test in \eqref{eq:test_fullstrong} \sid{with the 
  assumption $H_a^{(2)}$, that is, $S([a,b])=1.$}
Suppose the 2RV condition in \eqref{e:2RVmeas} holds with
a limiting signed
measure $\chi (\cdot)$ and \sid{ $A(t) \in
RV_{-\rho}$, $\rho>0$. Define $b_0(t) $  as in \eqref{e:defb0}},
so $b_0(t) \in RV_{1/(\alpha(1+\rho))}$ and $\alpha_0=\alpha(1+\rho)$.
 As before, $\{k_n\}$ is an intermediate sequence satisfying \eqref{eq:cond1}.
 Define 
\begin{align*}
\mu &:= \int_a^b xS(\dd x),\qquad
\sigma^2 := \int_a^b (x-\mu)^2 S(\dd x),
\end{align*}
and under strong dependence assumption $H^{(2)}_a$, we have
\begin{align}
\sqrt{k_n}\left(T_n - \frac{1}{\alpha}\right)
&\Rightarrow \frac{(1+\sigma^2/\mu^2)^{1/2}}{\alpha}\left(\int_0^1\frac{W(s)}{s}\dd s - W(1)\right)\nonumber
  \\
          &  \stackrel{d}{=}\frac{(1+\sigma^2/\mu^2)^{1/2}}{\alpha} W(1)
          \sim N\bigl(0,
            \frac{1}{\alpha^2}(1+\sigma^2/\mu^2)\bigr) .
            \label{eq:claimstrong}
\end{align}
\end{Theorem}
The proof of Theorem~\ref{thm:strong_full_HA} requires a
functional central limit theorem for row sums of a triangular array of
$\mathbb{D}[0,1]$-functions (\cite[Theorem 10.6]{pollard:1984}) that
generalizes the sequential result of \cite{hahn:1978}.
We give the formal proof for results in Theorem~\ref{thm:strong_full_HA} since 
it showcases the key proof steps for Theorem~\ref{thm:strong_full} as well.
\begin{proof}
Proceed by steps.\\
(1) First, employ the functional central limit theorem given in Theorem~4.1 of the supplement to show that
in $D(0,1]$,
\begin{align}
\label{eq:strong_step1}
\frac{\sqrt{k_n}}{(1+\sigma^2/\mu^2)^{1/2}}\left(\frac{1}{\mu k_n}\sum_{i=1}^n \Theta_i \epsilon_{R_i/b(n/k_n)}(t^{-1/\alpha},\infty)-t\right)
\Rightarrow W(t),
\end{align}
where $W(\cdot)$ is a standard Brownian motion.

We check all conditions in Theorem~4.1 of the supplement (details deferred to Section~4 of the supplement), and draw the conclusion that in $D(0,1]$,
\begin{align}
&\frac{(\mu^2+\sigma^2)^{-1/2}}{\sqrt{k_n}}\sum_{i=1}^n \left(\Theta_i \epsilon_{\frac{R_i}{b(n/k_n)}}(t^{-1/\alpha},\infty)-\EE\bigl(\Theta_i \epsilon_{\frac{R_i}{b(n/k_n)}}(t^{-1/\alpha},\infty)\bigr)\right)\nonumber\\
&\Rightarrow W(t).\label{eq:Xt}
\end{align}
Note that by 2RV using \eqref{e:expVer} or \eqref{e:expVerMeas}
plus the marginalized version for the distribution of $R_1$, we have
\begin{align}\label{eq:Theta2RV}
\sqrt{k_n}\left(\frac{n}{k_n}\EE\left(\Theta_1\ind_{\{R_1>b(n/k_n)y\}}\right)-\mu \frac{n}{k_n}\PP(R_1>b(n/k_n)y)\right)\to 0
\end{align}
locally uniformly for $y> 0$ and for $\{k_n\}$ satisfying
\eqref{eq:cond1}. Combining \eqref{eq:Theta2RV} with \eqref{eq:Xt} then completes the proof of \eqref{eq:strong_step1}.\\
(2)
Applying
the 
composition map $(x(t),c)\mapsto x(ct)$ from
{$D(0,\infty)\times (0,\infty)\mapsto D(0,\infty)$}, we get in $D(0,\infty)$,
\begin{align}\label{eq:conv_thetaR}
\frac{\sqrt{k_n}}{(1+\sigma^2/\mu^2)^{1/2}}\left(\frac{1}{\mu k_n}\sum_{i=1}^n \Theta_i\epsilon_{R_i/R_{(k_n)}}(y,\infty) 
- \left(y\frac{R_{(k_n)}}{b(n/k_n)} \right)^{-\alpha} \right)
\Rightarrow  W(y^{-\alpha}).
\end{align}
 Repeating a similar argument as in Step 2 of the proof for Theorem~\ref{thm:strong_full} (details included in Section 3 of the supplement), we are able to justify the application of
\[
x\mapsto \int_1^\infty \frac{x(s)}{s}\dd s,
\]
which further leads to 
\begin{align}
(\mu^2+\sigma^2)^{-1/2}\sqrt{k_n}&\left(\frac{1}{\mu k_n}\sum_{i=1}^{k_n}\Theta_i^*\log\frac{R_{(i)}}{R_{(k_n)}}-\frac{1}{\alpha}\left(\frac{R_{(k_n)}}{b(n/k_n)}\right)^{-\alpha}\right)\nonumber\\
&\Rightarrow \frac{1}{\alpha}\int_0^1 \frac{W(s)}{s}\dd s .\label{eq:conv_thetaLogR}
\end{align}
(3)
We are left to justify the convergence of 
\begin{align*}
\sqrt{k_n}(T_n-1/\alpha)& - \sqrt{k_n}\left(\frac{1}{\mu k_n}\sum_{i=1}^{k_n}\Theta_i^*\log\frac{R_{(i)}}{R_{(k_n)}}-\frac{1}{\alpha}\left(\frac{R_{(k_n)}}{b(n/k_n)}\right)^{-\alpha}\right)\\
&= \sqrt{k_n}\left(\left(\frac{1}{k_n}\sum_{i=1}^{k_n}\Theta_i^*\right)^{-1}-\frac{1}{\mu}\right)\frac{1}{k_n}\sum_{i=1}^{k_n}\Theta_i^*\log\frac{R_{(i)}}{R_{(k_n)}}\\
&+ \frac{\sqrt{k_n}}{\alpha}\left(\left(\frac{R_{(k_n)}}{b(n/k_n)}\right)^{-\alpha}-1\right).
\end{align*}
Note that by \eqref{eq:conv_thetaR},
\begin{align*}
\frac{1}{k_n}\sum_{i=1}^{k_n} \Theta_i^* = \frac{1}{k_n}\sum_{i=1}^n \Theta_i\epsilon_{R_i/R_{(k_n)}}(1,\infty) 
\convp \mu,
\end{align*}
and the convergence in \eqref{eq:conv_thetaLogR} gives
\begin{align*}
\frac{1}{k_n}\sum_{i=1}^{k_n}\Theta_i^*\log\frac{R_{(i)}}{R_{(k_n)}}\convp \frac{\mu}{\alpha}.
\end{align*}
Therefore, it suffices to consider the convergence of
\begin{align}\label{eq:RTheta_bar}
\frac{\sqrt{k_n}}{\alpha}\left(\left(\frac{R_{(k_n)}}{b(n/k_n)}\right)^{-\alpha}-\frac{1}{\mu k_n}\sum_{i=1}^{k_n}\Theta_i^*\right).
\end{align}

To prove the convergence of \eqref{eq:RTheta_bar},
we first use Vervaat's inversion
(\cite{vervaat:1972b,dehaan:ferreira:2006,resnickbook:2007})  
to obtain the convergence of the inverse of 
\[
\eta_n(\cdot) := \frac{1}{\mu k_n}\sum_{i=1}^n\Theta_i\epsilon_{R_i/b(n/k_n)}\bigl((\cdot)^{-1/\alpha},\infty\bigr).
\]
Note that
\begin{align}
\eta_n^\leftarrow(t) &= \inf\left\{s: \frac{1}{k_n}\sum_{i=1}^n\Theta_i\epsilon_{R_i/b(n/k_n)}\bigl(s^{-1/\alpha},\infty\bigr)\ge \mu t\right\}\nonumber\\
&= \left(\sup\left\{y: \frac{1}{k_n}\sum_{i=1}^n\Theta_i\epsilon_{R_i/b(n/k_n)}\bigl(y,\infty\bigr)\ge \mu t\right\}\right)^{-\alpha}.\label{eq:nu_inverse}
\end{align}
Then with 
\[
M_n(t):=\inf\left\{m\ge 1: \frac{1}{k_n}\sum_{i=1}^n\Theta_i\epsilon_{R_i/b(n/k_n)}\left(\frac{R_{(m)}}{b(n/k_n)},\infty\right)\ge \mu t\right\},
\]
the inverse function in \eqref{eq:nu_inverse} becomes
\begin{align*}
\eta_n^\leftarrow(t) = \left(\frac{R_{(M_n(t))}}{b(n/k_n)}\right)^{-\alpha}.
\end{align*}
Applying Vervaat's lemma
(\cite{vervaat:1972b,dehaan:ferreira:2006,resnickbook:2007}) 
  gives the joint convergence in $D(0,1]\times D(0,1]$:
\begin{align}\label{eq:joint_conv}
\frac{\sqrt{k_n}}{(1+\sigma^2/\mu^2)^{1/2}}\left(\eta_n(t)-t,\,\eta^\leftarrow_n(t)-t\right)
\Rightarrow \left(W(t), -W(t)\right).
\end{align}
Since for $t=\frac{1}{\mu k_n}\sum_{i=1}^{k_n}\Theta_i^*$, 
\begin{align*}
&M_n\left(\frac{1}{\mu k_n}\sum_{i=1}^{k_n}\Theta_i^*\right) \\
&= \inf\left\{m\ge 1: \frac{1}{k_n}\sum_{i=1}^n\Theta_i\epsilon_{R_i/b(n/k_n)}\left(\frac{R_{(m)}}{b(n/k_n)},\infty\right)\ge \frac{1}{k_n}\sum_{i=1}^{k_n}\Theta_i^*\right\}
= k_n,
\end{align*}
then \eqref{eq:joint_conv} gives
\begin{align}\label{eq:conv_inversion}
\sqrt{k_n}\left(\left(\frac{R_{(k_n)}}{b(n/k_n)}\right)^{-\alpha}-\frac{1}{\mu k_n}\sum_{i=1}^{k_n}\Theta_i^*\right)\Rightarrow -(1+\sigma^2/\mu^2)^{1/2}W(1).
\end{align}
Combining \eqref{eq:conv_inversion} with \eqref{eq:conv_thetaLogR} shows that
\[
\sqrt{k_n}\left(T_n-1/\alpha\right)\Rightarrow \frac{(1+\sigma^2/\mu^2)^{1/2}}{\alpha}\left(\int_0^1\frac{W(s)}{s}\dd s - W(1)\right),
\]
thus verifying \eqref{eq:claimstrong}.
\end{proof}

\section{Implementation of Testing}\label{sec:boot}
Applying the test statistics to data requires estimating
a minimal length interval $[a,b]$ containing
the support of the angular measure. 
On the one hand, choosing an unnecessarily  wide
interval $[a,b]$ leads $D_n^*$ to conclude $S([a,b])=1$ but only shows the support is a subset of
$[a,b]$.
{Also making $[a,b]$ too wide may mean there are few points in $[0,1]\setminus [a,b]$, so that even if
the true support of $S$ is $[0,1]$, we could falsely accept the
existence of strong dependence.} 
On the other hand, fixing an excessively narrow interval $[a,b]$ may
lead to $D_n^*$ inaccurately
rejecting
existence of strong dependence.

\sid{We begin with a method for estimating $a,b$ and then proceed to
bootstrap methods for implementing the tests. This is followed in
Sections \ref{subsec:sim} and \ref{sec:GetReal} by illustrations using
simulated and real data.}

\subsection{Methodology}
\subsubsection{Estimating $[a,b]$}\label{subsub:estimate}

\sid{We estimate $a,b$ as the minimizer of an objective function
$g_n(a,b)$ subject to the constraint $0\le a\le b\le 1$ where}
\begin{align}
g_n(a,b):= (b-a) + \sqrt{k_n} \left\vert D_n^* - H_{k_n,n}\right\vert\label{eq:opt}
\end{align}
The first part of the objective function, $b-a$, favors a narrow
interval $[a,b]$ the second part requires  a wide
enough interval $[a,b]$ so that
$\left\vert D_n^* - H_{k_n,n}\right\vert\approx 0$.
Hence, by minimizing $g_n$, we obtain an estimated interval
$[\hat{a},\hat{b}]$ of reasonable length and satisfying $\left\vert
  D_n^* - H_{k_n,n}\right\vert\approx 0$. In practice, the {\tt constrOptim}
function in R suffices for the minimization.

Theorem~\ref{thm:consistency} gives the consistency of $\hat{a}$ and $\hat{b}$ for $\alpha>1$.

\begin{Theorem}\label{thm:consistency}
 Suppose the support of $S$ is $[a,b]$,    $\alpha>1$ and the intermediate sequence $\{k_n\}$ satisfies
 \eqref{eq:cond1}.
 Let $\hat{a}$ and $\hat{b}$ be the minimizer of
\eqref{eq:opt}. Then,
\[
\hat{a}\convp a,\qquad \hat{b}\convp b,
\]
as $n\to\infty$.
\end{Theorem}
In fact, the consistency result in Theorem~\ref{thm:consistency} also holds if 
we redefine
\[
g_n (s,t) = \sid{(t-s)} + \lambda\sqrt{k_n} \left\vert D_n^* - H_{k_n,n}\right\vert,
\]
for some $\lambda>0$. In Section~\ref{subsec:sim}, we exemplify the
value of this flexibility.

\begin{proof}
  To shorten the proof and ease notation we make the simplifying
  assumption that we know $b=1$. Define,
  \begin{align*}
&\mathbb{C}_{s}=\mathbb{C}_{s,1} = \{(x,y)\in R_+^2: s\le x/(x+y)\},\\
&d^*_{s}\bigl((x,y)\bigr)=d^*_{s,1}\bigl((x,y)\bigr) = d^*\bigl((x,y), \mathbb{C}_{s}\bigr)
    =\left\{y-(s^{-1}-1)x\right\}^+,\qquad 0\le s\le  1.\\
&g_n(s)=g_n(s,1)       =(1-s) + \frac{1}{\sqrt{k_n}}\sum_{i=1}^{k_n}\frac{d^*_{s}\bigl((X_i^*,Y_i^*)\bigr)}{R_{(k_n)}}\log\frac{R_{(i)}}{R_{(k_n)}}.
\end{align*}
Note that $g_n(\cdot)$ is concave in $s$.
We prove Theorem~\ref{thm:consistency} by showing that for some
$\epsilon>0$, 
and $\mathcal{I}_\epsilon := [0,a-\epsilon]\cup
[a+\epsilon,1],$
\begin{align}\label{eq:gn_diff}
\PP\left(\inf_{s \in \mathcal{I}_\epsilon}
  (g_n(s)-g_n(a))>2\epsilon\right)\to 1, \quad (n\to\infty).
\end{align}

Since $d^*_{s}\bigl((x,y)\bigr)$ is increasing in $s$,
Theorem~\ref{thm:test_strong} ensures that {(see \eqref{e:small})} 
\[
\sup_{s\in [0,a] } \frac{1}{\sqrt{k_n}}\sum_{i=1}^{k_n}
\frac{d^*_{s,t}\bigl((X_i^*,Y_i^*)\bigr)}{R_{(k_n)}}\log\frac{R_{(i)}}{R_{(k_n)}}
\convp 0, 
\]
and therefore,
\begin{align}
\label{eq:consistent_b1}
  \lim_{n\to\infty}
  \PP\left(\inf_{s\leq a-\epsilon}(g_n(s)-g_n(a))>2\epsilon\right) =1.
\end{align}

If $s\geq a+\epsilon,$ replace division by $\sqrt{k_n}$ with division
by $k_n$ and 
the definitions of $(R_i, \Theta_i)$ give
\begin{align}
&\frac{1}{k_n}\sum_{i=1}^{k_n}\frac{d^*_{s}\bigl((X_i^*,Y_i^*)\bigr)}{R_{(k_n)}}\log\frac{R_{(i)}}{R_{(k_n)}}
= \frac{1}{k_n}\sum_{i=1}^{k_n}\left\{1-s^{-1}\Theta^*_i \right\}^+
                \frac{R_{(i)}}{R_{(k_n)}}\log\frac{R_{(i)}}{R_{(k_n)}},\nonumber
  \\
  \intertext{and since $R_{(i)}\geq R_{(k_n)}$, this is bounded
  below by}
  &\ge \frac{1}{k_n}\sum_{i=1}^{k_n}\left\{1-s^{-1}\Theta^*_i
    \right\}^+\log\frac{R_{(i)}}{R_{(k_n)}}. \label{eq:limit_dstar}
\end{align}
We show this converges in probability to a positive constant
$L(a,s)>0$ and thus
\begin{align}
\label{eq:consistent_b2}
\frac{1}{\sqrt{k_n}}&\sum_{i=1}^{k_n}\frac{d^*_{s}\bigl((X_i^*,Y_i^*)\bigr)}{R_{(k_n)}}\log\frac{R_{(i)}}{R_{(k_n)}}\convp \infty.
\end{align}
Combining \eqref{eq:consistent_b2} with \eqref{eq:consistent_b1} completes the proof of \eqref{eq:gn_diff}.

Returning to the expression in \eqref{eq:limit_dstar}, write it as
\begin{equation}\label{e:sharpsharp}
  \iint_{\{(r,\theta): r>1,\theta \in [0,1]\}}
  \Bigl(1-\frac{\theta}{s}\Bigr)^+ \log r \frac 1k \sum_{i=1}^{k_n}
  \epsilon_{(\Theta_i^*, R_{(i)} / R_{(k)} )  }(d\theta,dr).
\end{equation}
On $[0,1]\times (1,\infty)$,
$$
\frac 1k \sum_{i=1}^{k_n}
  \epsilon_{(\Theta_i^*, R_{(i)} / R_{(k)} )  }(d\theta,dr) \to
  S(d\theta)\nu_\alpha(dr)$$
  and the right side is a probability
  measure on $[0,1]\times (1,\infty)$, so using familiar weak
  convergence arguments in \eqref{e:sharpsharp} we get convergence to
  \begin{align*}
    L(a,s):=&   \iint_{\{(r,\theta): r>1,\theta \in [0,1]\}}
              \Bigl(1-\frac{\theta}{s}\Bigr)^+ \log r S(d\theta)\nu_\alpha(dr)\\
    \intertext{and after some Fubini justified manipulations this is}
    =&\frac{1}{\alpha} \int_a^s \Bigl(1-\frac{\theta}{s}\Bigr) S(d\theta).
  \end{align*}

Remember $s>a+\epsilon$ and we verify $L(a,s)$ is positive. If not,
$L(a,s)=0$ and $1-\theta/s=0 $ or $\theta=s$ for almost all (with
respect to $S(\cdot)$) $\theta
\in [a,,a+\epsilon]$ and this means $S[a,a+\epsilon]=0$, thus
contradicting  $a$ being  in the support of $S(\cdot)$. So $L(a,s)>0.$
\end{proof}

\subsubsection{Bootstrap methods}
Formulating tests based on either Theorem~\ref{thm:test_strong} or \ref{thm:strong_full} requires knowing
the values of $\alpha, a, b$, which, however, is unlikely to be true for real datasets. 
\twang{
Substitution methods suggest  replacing $\alpha$ with the
corresponding Hill estimator, $1/H_{k_n,n}$ and investigating the
effect on the limit distribution but this
will not work here due to  \eqref{e:small}.
In the sequel, we \sid{propose} bootstrap methods to implement the
proposed tests \sid{and try the approach on  simulated and real datasets.}
We do not \sid{formally justify} the bootstrap method--this is
left for the future--but the numerical experiments suggest its applicability.
}

Suppose we take $m_n\approx n/k_n$ so that $m_n/n\to 0$ \sid{and
  $m_n\to \infty$}. Let $\{I_1(n),\ldots,I_{m_n}(n)\}$ be iid discrete uniform random variables on $\{1,\ldots,n\}$,
independent from $\{(X_i,Y_i): i\ge 1\}$.
We construct \sid{a} bootstrap resample of size $m_n$ by
\[
(X_{I_j(n)},Y_{I_j(n)}), \qquad j=1,\ldots,m_n.
\]
Define $R^\text{boot}_{(i)}$ as the $i$-th largest order statistic among
$\{R_{I_j(n)}\equiv X_{I_j(n)}+Y_{I_j(n)}:1\le j\le m_n\}$, and let 
$(X^*_{I_i(n)},Y^*_{I_i(n)})$ be the pair of random variables such that $X^*_{I_i(n)}+Y^*_{I_i(n)}\equiv R^\text{boot}_{(i)}$.
\medskip

\noindent\textbf{(1) Test $H_0^{(1)}$.}
For the test in \eqref{eq:test_strong}, we
first solve \eqref{eq:opt} \sid{using the whole sample} to estimate the support of the angular measure, $[\hat{a},\hat{b}]$ from the sample.
Then we 
obtain $\mathbb{C}_{\hat{a},\hat{b}}:= \{(x,y)\in \RR_+^2: \hat{a}\le x/(x+y)\le \hat{b}\}$
and
\begin{align*}
\widehat{D}^*_{m_n}=&\frac{1}{{k_n}}\sum_{i=1}^{k_n}
                      \left(1+\frac{d^*\bigl((X^*_{I_j(n)},Y^*_{I_j(n)}),{\mathbb{C}}_{\hat{a},\hat{b} }\bigr)}{R^\text{boot}_{(k_{m_n})}}\right)\log\frac{R^\text{boot}_{(i)}}{R^\text{boot}_{(k_{m_n})}}.
\end{align*}
\twang{Conditioning \sid{on} the original sample, 
we \sid{presume} from Theorem~\ref{thm:test_strong} $\sqrt{k_{m_n}}\left(\widehat{D}^*_{m_n}-H_{k_n,n}\right)$ is approximately normal with mean 0 and variance $H_{k_n,n}^2$
for large $n$.
}
Therefore,
we reject $H^{(1)}_0$ in \eqref{eq:test_strong} if 
\begin{align}
\label{eq:bootD}
\left\vert \widehat{D}^*_{m_n}-H_{k_n,n}\right\vert > 1.96 \frac{H_{k_n,n}}{\sqrt{k_{m_n}}}.
\end{align}
\sid{In practice we would generate $B$ bootstrap samples and reject if more
than 5\% satisfy \eqref{eq:bootD}.}

\noindent\textbf{(2) Full vs strong dependence.}
For the test in \eqref{eq:test_fullstrong},  generate $B$ bootstrap
resamples \sid{indexed by $t=1,\dots,B$.}
\sid{For each $t$, let $R^\text{boot}_{(i),t}$ be the $i$-largest
  order statistic in the $t$-th resample;
$\Theta^*_{i,t}$ is the corresponding concomitant.}
Compute the corresponding test statistics for each resample, 
\[
T_{m_n}^{(t)} = \frac{\sum_{i=1}^{k_{m_n} }\Theta_{i,t}^*\log\frac{R^\text{boot}_{(i),t}}{R^\text{boot}_{(k_{m_n}), t}}}{\sum_{i=1}^{k_{m_n}}\Theta_{i,t}^*},
\qquad  t=1,\ldots, B.
\]
Based on Theorem~\ref{thm:strong_full},
we presume under $H_0^{(2)}$ that conditional on the original sample, 
$\sqrt{k_{m_n}}\left(T_{m_n}^{(t)} - H_{k_n,n}\right)$
is approximately normal with mean 0 and variance $H_{k_n,n}^2$
for large $n$.
Using all $B$ resamples, we obtain the bootstrap estimate of the standard error of $T_n$:
\[
SE_\text{boot}(m_n) := \left(\frac{1}{B-1}\sum_{t=1}^B \left(T_{m_n}^{(t)}-\bar{T}_{m_n}\right)^2\right)^{1/2},
\]
where $\bar{T}_{m_n} = \frac{1}{B}\sum_{t=1}^B T_{m_n}^{(t)}$.
Then we reject $H^{(2)}_0$ in \eqref{eq:test_fullstrong} if 
\[
k_{m_n}\frac{SE^2_\text{boot}(m_n)}{H_{k_n,n}^2} > \chi^2_{.95, B-1}/(B-1),
\]
where $\chi^2_{.95, B-1}$ denotes the 95\% quantile of a chi-square distribution with $B-1$ degrees of freedom.

\noindent\textbf{(3) Strong vs weak dependence.}
When testing for strong vs weak dependence, we rely on Theorem~\ref{thm:strong_full_HA} and
define $\widetilde{\Theta}_i := \Theta_i \ind_{\{\Theta_i \in [a,b]\}}$, $\widetilde{R}_i := R_i \ind_{\{\Theta_i \in [a,b]\}}$.
Let $\widetilde{\Theta}_i^*$ be the concomitant of $\widetilde{R}_{(i)}$, and by assuming $0/0\equiv 1$ we define also 
\begin{align}
\label{eq:defTn_ab}
\widetilde{T}_n := \frac{\sum_{i=1}^{k_n}\widetilde{\Theta}_i^*\log\left(\frac{\widetilde{R}_{(i)}}{\widetilde{R}_{(k_n)}}\vee 1\right)}{\sum_{i=1}^{k_n}\widetilde{\Theta}_i^*}.
\end{align}
For $[a,b]\subsetneq [0,1]$, we want to test strong vs weak dependence, i.e.
\begin{align}
\label{eq:strong_weak}
H^{(3)}_0: \text{Support of $S(\cdot)=[a,b]$} \qquad \text{v.s.}\qquad H^{(3)}_a:\text{Support of $S(\cdot)=[0,1]$}.
\end{align}
Under $H^{(3)}_0$, $\widetilde{T}_n$ must have the same asymptotic distribution as $T_n$.
Here we apply the bootstrap method again to test whether $T_n$ and $\widetilde{T}_n$ have the same asymptotic variance.
Again estimate $[\hat{a},\hat{b}]$ from \eqref{eq:opt}. To obtain the $t$-th resample, we generate $\{I_{1,t}(n),\ldots,I_{m_n,t}(n)\}$ iid discrete uniform random variables on $\{1,\ldots,n\}$, and compute
\[
\widetilde{\Theta}_{i,t} := \Theta_{I_{i,t}(n)} \ind_{\{\Theta_{I_{i,t}(n)} \in [\hat{a},\hat{b}]\}}
,\qquad \widetilde{T}_{m_n}^{(t)} = \frac{\sum_{i=1}^{k_{m_n}}\widetilde{\Theta}_{i,t}^*\log\left(\frac{\widetilde{R}_{(i),t}}{\widetilde{R}_{(k_{m_n}),t}}\vee 1\right)}{\sum_{i=1}^{k_{m_n}}\widetilde{\Theta}_{i,t}^*}.
\]
We repeat the bootstrap resampling scheme twice to obtain $T_{m_n}^{(1)},\ldots, T_{m_n}^{(B)}$, $\widetilde{T}_{m_n}^{(1)},\ldots, \widetilde{T}_{m_n}^{(B)}$, and reject $H^{(3)}_0$ if
\[
\frac{\frac{1}{B-1}\sum_{t=1}^B \left(T_{m_n}^{(t)}-\bar{T}_{m_n}\right)^2}{\frac{1}{B-1}\sum_{s=1}^B \left(\widetilde{T}_{m_n}^{(t)}-\bar{\widetilde{T}}_{m_n}\right)^2}
> F_{0.975, B-1, B-1}\quad\text{or}\quad < F_{0.025, B-1, B-1},
\]
where $\bar{\widetilde{T}}_{m_n} = \frac{1}{B}\sum_{t=1}^B \widetilde{T}_{m_n}^{(t)}$ and 
$F_{p, B-1, B-1}$ denotes the $100p\%$-percentile of an $F$-distribution with numerator and denominator degrees of freedom both equal to $B-1$.

\subsection{Simulation study}\label{subsec:sim}
\textbf{Example 1.}
Consider a simulated data example as below. Set $a=0.25$, $b=0.75$. Suppose $R_1\sim\text{Pareto}(2)$, $R_2\sim\text{Pareto}(4)$,
$Z\sim \text{Beta}(0.05,0.1)$, $\Theta_2\sim \text{Unif}([0,1]\setminus [a,b])$, and $B\sim \text{Bernoulli}(0.5)$.
 Assume the random variables are all independent, and let $\Theta_1 := a + (b-a)Z$.
Now define the vector $(X,Y)$ as 
\begin{align*}
X &:= BR_1\Theta_1 + (1-B)R_2\Theta_2\\
Y &:= BR_1(1-\Theta_1) + (1-B)R_2(1-\Theta_2).
\end{align*}
By construction, $(X,Y)$ is MRV on $\RR^2_+\setminus \{\origin\}$
with tail parameter $\alpha = 2$.
The second order condition \eqref{e:2RVmeas} also holds since
\[
\frac{1}{t^{-1}}\left(t\PP\left[\left(\frac{R}{b(t)},\Theta\right)\in\cdot\right]-p\nu_2\times\PP[\Theta_1\in\cdot]\right)
\to (1-p)\nu_4\times\PP[\Theta_2\in\cdot].
\]
Furthermore, for 
\[
\mathbb{C}_{a,b} = \{(x,y)\in \RR^2_+: y/3\le x\le 3y\},
\]
the vector $(X,Y)$ has HRV on $\RR^2_+\setminus \mathbb{C}_{a,b}$
with tail parameter $\alpha_0 = 4$.

We generate $n = 30{,}000$ iid samples from the distribution of
$(X,Y)$.
The scatter plot is in the left panel of Figure~\ref{fig:scatter} and
illustrates
the dependence structure.
Thresholding with $k_n=100$ yields the histogram of
angles in the right panel of Figure~\ref{fig:scatter}; this also
describes the dependence structure of $(X,Y)$ and
based on high values of $\vert x\vert+\vert y\vert$, shows that the support 
of the angular measure is $[0.25,0.75]$. Based on this sample, the Hill estimate with $k_n=100$ is $H_{k_n,n}=0.473$.

\begin{figure}
\centering
\includegraphics[width=\textwidth]{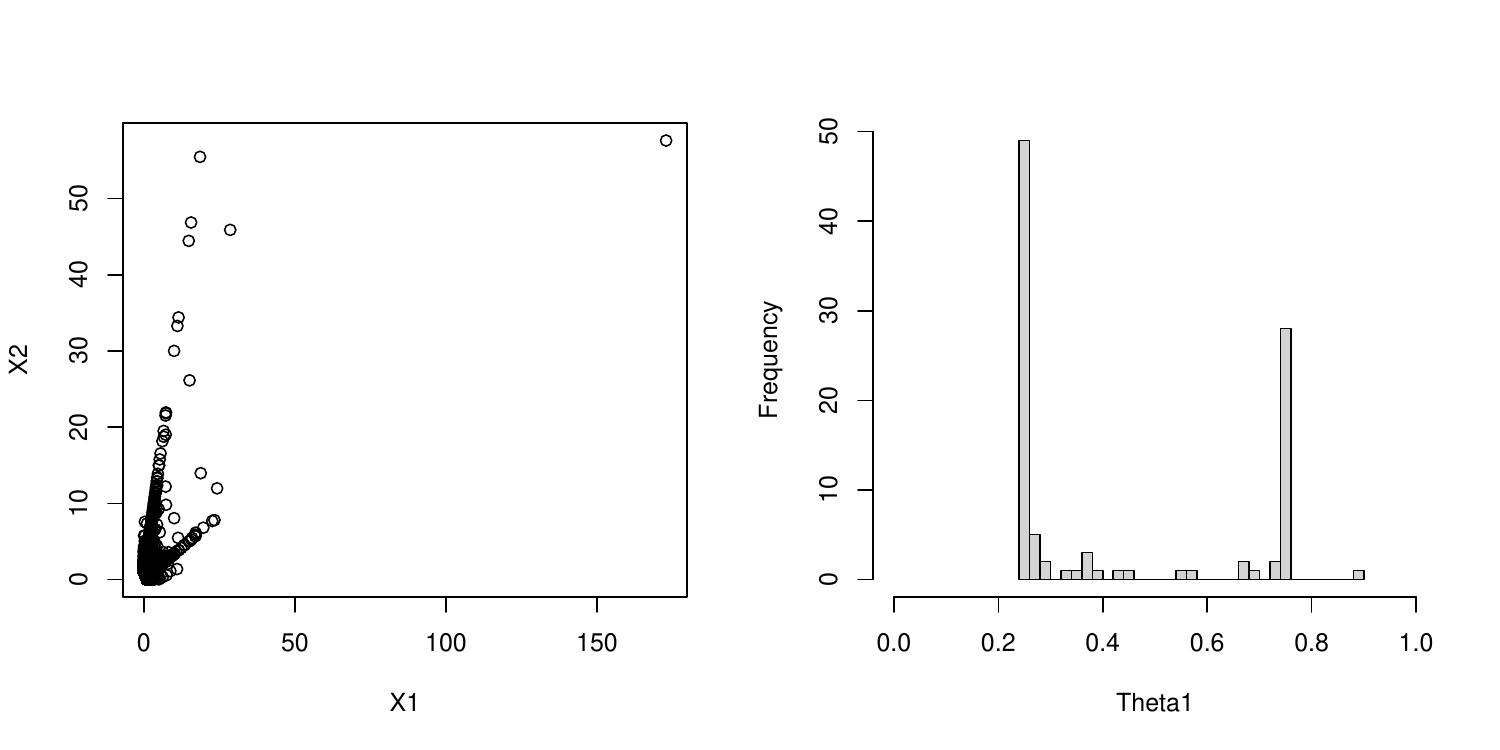}
\caption{Simulated data example 1. Left: Scatter plot of $30{,}000$ data points. Histogram of $\theta_1$ with $k_n=100$.}\label{fig:scatter}
\end{figure}

To estimate $[a,b]$, we solve the optimization problem in
\eqref{eq:opt}
\sid{using the {\tt constrOptim} function in R}, and Table~\ref{table:lambda_ex1} shows the estimated $[\hat{a},\hat{b}]$ for 
different choices of the tuning parameter $\lambda$. We see that for a bimodal histogram as in the right panel of Figure~\ref{fig:scatter},
solving \eqref{eq:opt} provides good estimates for $a,b$, invariant to different choices of the tuning parameter.
\begin{table}[h]\centering
\begin{tabular}{c|c|c|c|c|c}
$\lambda\sqrt{k_n}$ & $\sqrt{k_n}$ & $2\sqrt{k_n}$ & $2^2\sqrt{k_n}$ & $2^3\sqrt{k_n}$ & $2^4\sqrt{k_n}$\\
\hline
$[\hat{a}, \hat{b}]$ & $[0.25, 0.75]$ & $[0.25, 0.75]$ & $[0.250, 0.750]$ & $[0.25, 0.75]$ & $[0.25, 0.75]$
\end{tabular}
\caption{Estimated $[\hat{a},\hat{b}]$ for 
different choices of the tuning parameter $\lambda$ when $[a,b]=[0.25, 0.75]$.}\label{table:lambda_ex1}
\end{table}

Next, we set $m_n=500$, $k_{m_n}=25$, and generate $B=2{,}000$ bootstrap resamples to test
\[
H^{(1)}_0: \, S([0.25,0.75]) = 1,\qquad H^{(1)}_a:\, S([0.25,0.75]) < 1.
\]
For each bootstrap resample, we compute the corresponding $\widehat{D}_{m_n}^*$ and use \eqref{eq:bootD} to decide whether to reject $H^{(1)}_0$ or not.
Among the 2{,}000 bootstrap resamples generated, the rejection rate is 0.045, indicating we shall accept $H^{(1)}_0$.
In addition, applying the boostrap testing procedure for $H_0^{(3)}$, we have
\begin{align*}
\frac{\frac{1}{B-1}\sum_{t=1}^B \left(T_{m_n}^{(t)}-\bar{T}_{m_n}\right)^2}{\frac{1}{B-1}\sum_{t=1}^B \left(\widetilde{T}_{m_n}^{(t)}-\bar{\widetilde{T}}_{m_n}\right)^2}
&= 1.077 \in [0.916, 1.092] \\
&= [F_{0.025, 1999, 1999},F_{0.975, 1999, 1999}].
\end{align*}
This confirms the existence of strong dependence on $\RR_+^2\setminus\mathbb{C}_{a,b}$. 
To check further the existence of full dependence, we compute $SE_\text{boot}(m_n) = 0.546$ based on the bootstrap resamples, which gives
\[
k_{m_n}\frac{SE^2_b(m_n)}{H_{k_n,n}^2} = 1.336 > \chi^2_{.95, 1999}/1999 = 1.053.
\]
Hence, we reject the full dependence hypothesis in $H_0^{(2)}$.

\textbf{Example 2.} In fact, whether to reject $H_0^{(2)}$ also depends on the variability of $\Theta_1$. For instance, 
suppose instead the random variable $Z$ in the previous example follows a $\text{Beta}(1,2)$ distribution, and all other
assumptions remain identical to Example 1. 
The scatter plot and the histogram of angles are given in Figure~\ref{fig:scatter2}.

\begin{figure}
\centering
\includegraphics[width=\textwidth]{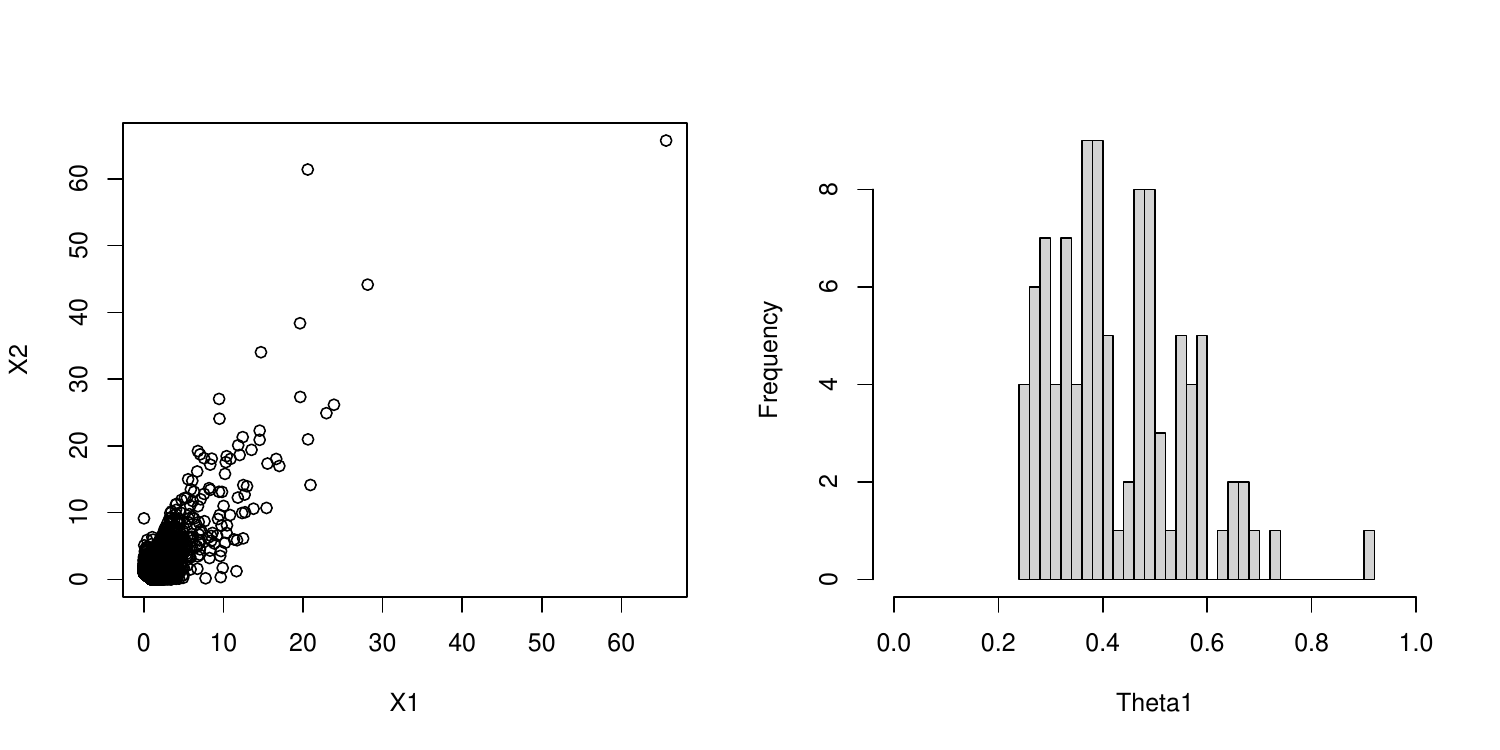}
\caption{Simulated data example 2. Left: Scatter plot of $30{,}000$ data points.Histogram of $\theta_1$ with $k_n=100$.}\label{fig:scatter2}
\end{figure}

We again estimate $[a,b]$ by solving the optimization problem in \eqref{eq:opt}, and Table~\ref{table:lambda_ex2} shows the estimated $[\hat{a},\hat{b}]$ for 
different choices of the tuning parameter $\lambda$. We see that here
the estimation procedure provides \sid{an} accurate estimate for $\hat{a}$
across  
all chosen values of $\lambda$, but estimated values of $\hat{b}$ are all smaller than $b=0.75$. Overall, $\lambda = 2^4$ provides the most accurate estimates.
\begin{table}[h]\centering
\begin{tabular}{c|c|c|c|c|c}
$\lambda\sqrt{k_n}$ & $\sqrt{k_n}$ & $2\sqrt{k_n}$ & $2^2\sqrt{k_n}$ & $2^3\sqrt{k_n}$ & $2^4\sqrt{k_n}$\\
\hline
$[\hat{a}, \hat{b}]$ & $[0.251, 0.554]$ & $[0.251, 0.596]$ & $[0.251, 0.597]$ & $[0.251, 0.654]$ & $[0.251, 0.670]$
\end{tabular}
\caption{Estimated $[\hat{a},\hat{b}]$ for 
different choices of the tuning parameter $\lambda$ when $[a,b]=[0.25, 0.75]$.}\label{table:lambda_ex2}
\end{table}

For this simulated dataset, we now proceed with the true values of $a$ and $b$.
Following the testing procedure 
as in the previous example, we see that out of $2{,}000$ bootstrap resamples, the overall rejection rate for $H^{(1)}_0$ is 0.0465, 
and the boostrap method for testing $H_0^{(3)}$ returns a test statistic 
\begin{align*}
\frac{\frac{1}{B-1}\sum_{t=1}^B \left(T_{m_n}^{(t)}-\bar{T}_{m_n}\right)^2}{\frac{1}{B-1}\sum_{t=1}^B \left(\widetilde{T}_{m_n}^{(t)}-\bar{\widetilde{T}}_{m_n}\right)^2}
&= 1.035 \in [0.916, 1.092] \\
&= [F_{0.025, 1999, 1999},F_{0.975, 1999, 1999}],
\end{align*}
confirming
the existence of strong dependence. However, these bootstrap resamples give $SE_\text{boot}(m_n) = 0.508$, thus giving
\[
k_{m_n}\frac{SE^2_\text{boot}(m_n)}{H_{k_n,n}^2} = 0.939 < \chi^2_{.95, 1999}/1999 = 1.053.
\]
This makes us fail to reject the full dependence hypothesis in $H_0^{(2)}$.
In fact, since $\Theta_1$ follows a Beta$(1,2)$ distribution, then we have $\mu = 0.417$ and $\sigma^2=0.014$, which leads to $(\sigma^2/\mu^2+1)^{1/2} = 1.039<1.053$.
Hence, the small variation in the underlying distribution of
$\Theta_1$ may lead to false acceptance of $H_0^{(2)}$. If replacing $[a,b]$ with the estimated $[\hat{a},\hat{b}] = [0.251, 0.670]$, 
we obtain the same conclusion.

So for this example, we fail to reject all three null hypotheses, $H_0^{(i)}$, $i=1,2,3$, and the third failure to reject is in error.
One way to fix this is to check the proportion of resamples among $t=1,\ldots,B$, which have
\begin{align}
\label{eq:full_test}
\left\vert T_{m_n}^{(t)}-H_{k_n,n}\right\vert > 1.96 \frac{H_{k_n,n}}{\sqrt{k_{m_n}}},
\end{align}
so that we will reject $H_0^{(2)}$.
Here the reject rate among the $2{,}000$ resamples is 0.059, which
suggests
we should reject the full dependence hypothesis; {this brings us to the
correct decision for this (simulated) data.}

\subsection{Real data examples}\label{sec:GetReal}
\begin{figure}
\centering
\includegraphics[width=\textwidth]{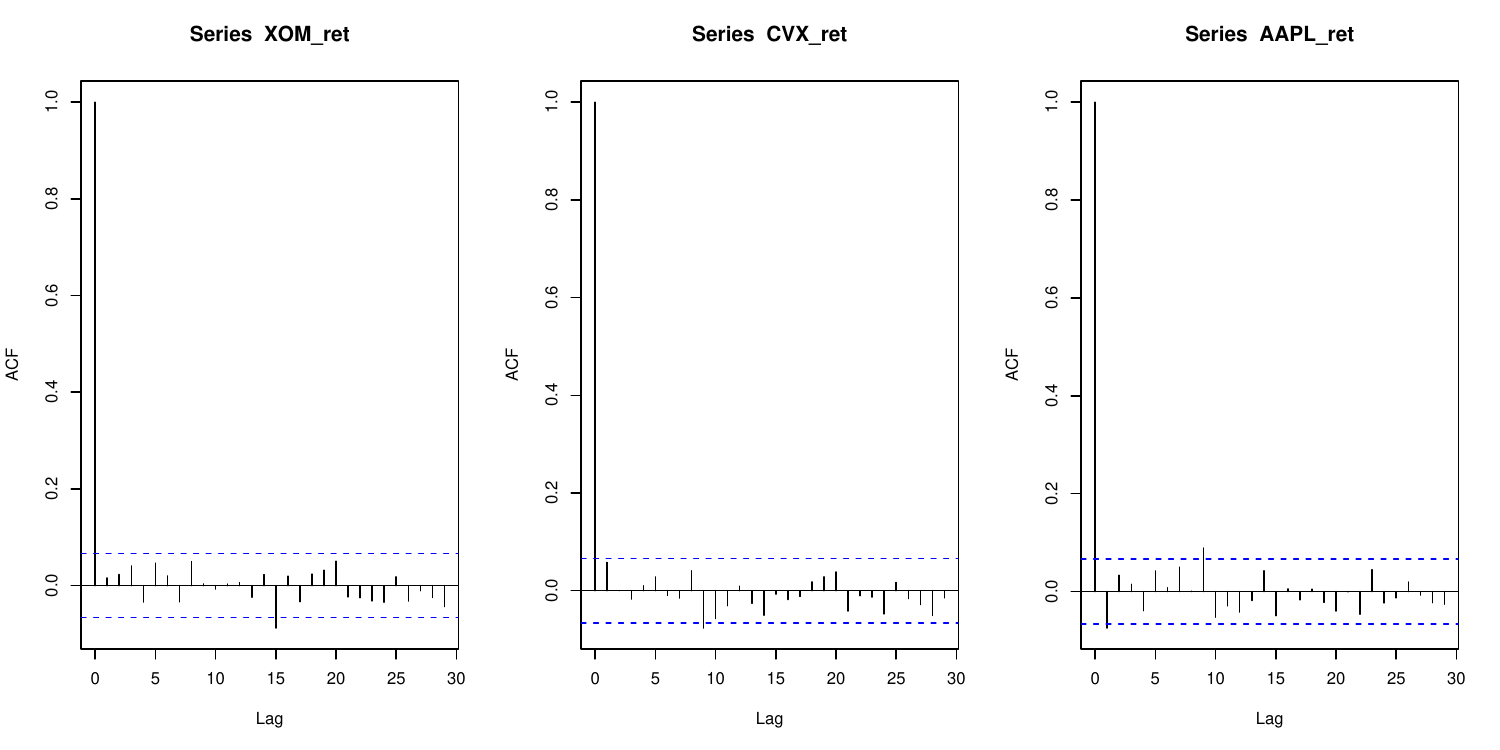}
\caption{Acf plots for the log returns of every-other-day stock prices.}\label{fig:acf}
\end{figure}

We now consider the application of the bootstrap method to real data. We download the daily adjusted stock prices of Chevron (CVX), Exxon (XOM) and Apple (AAPL) during the time 
period from January 04, 2016 to December 30, 2022. To remove the possible serial dependence of stock returns, we compute the log returns of these three stocks using
their every-other-day prices. 
The acf plots in Figure~\ref{fig:acf} show little serial dependence for all three stocks.
This leads to a reduced dataset of $n=880$ observations for each stock. 

\subsubsection{CVX vs XOM}
\begin{figure}
\centering
\includegraphics[width=\textwidth]{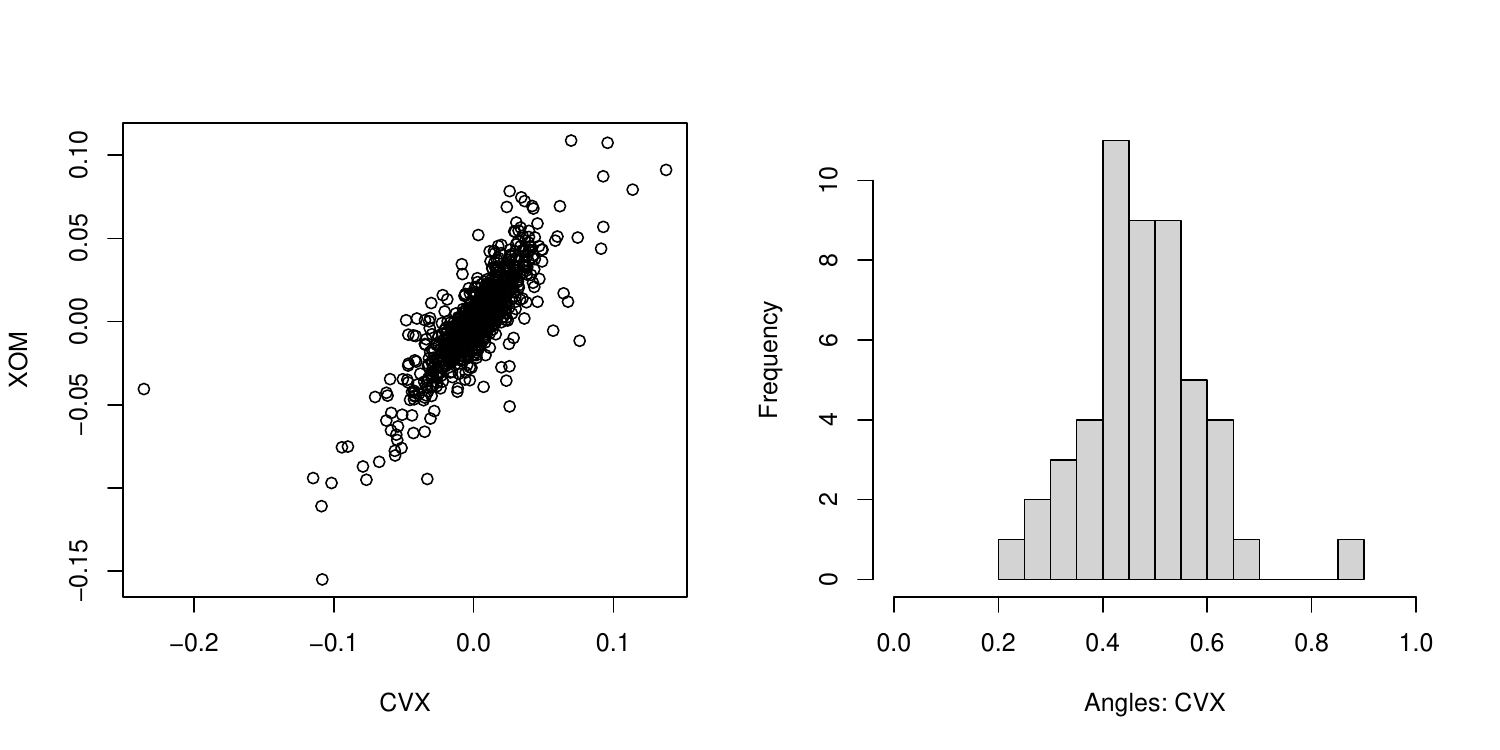}
\caption{CVX vs XOM. Left: Scatter plot of CVX and XOM returns. Right: Histogram of angles (absolute returns of CVX) with $k_n=100$.}\label{fig:CVX_XOM}
\end{figure}

In the left panel of Figure~\ref{fig:CVX_XOM}, we present the scatter plot of the log returns of CVX and XOM.
To understand the dependence structure between absolute log returns of CVX and XOM, we also graph the histogram of $\vert x\vert/(\vert x\vert+\vert y\vert)$
in the right panel of Figure~\ref{fig:CVX_XOM}, where the threshold is chosen as $k_n=100$.

The corresponding Hill estimate gives $\hat{\alpha}=1/H_{k_n,n} = 1/0.342=2.926$. 
By setting $\lambda = 4$, we obtain estimates
$\hat{a} = 0.336$ and $\hat{b}=0.853$ (estimates remain the same for $\lambda\ge 4$) as well as
\[
\frac{1}{k_n}\sum_{i=1}^{k_n}\Theta_i^* = 0.482 \equiv \widehat{\theta}_0.
\]
 Then we generate 
2{,}000 bootstrap resamples with $m_n = 200$ and $k_{m_n} = 20$ to test
\[
H^{(1)}_0: \, S([0.336,0.853]) = 1,\qquad H^{(1)}_a:\, S([0.336,0.853]) < 1.
\]
For each bootstrap resample, we compute the corresponding test
statistic $\widehat{D}_{k_{m_n}}^*$, and see that only 2.1\% of the
2000 bootstrap trials  reject $H^{(1)}_0$. 
In addition, consider strong vs weak dependence (i.e. $H^{(3)}_0$ vs $H^{(3)}_a$), and calculate
\begin{align*}
\frac{\frac{1}{B-1}\sum_{t=1}^B \left(T_{m_n}^{(t)}-\bar{T}_{m_n}\right)^2}{\frac{1}{B-1}\sum_{t=1}^B \left(\widetilde{T}_{m_n}^{(t)}-\bar{\widetilde{T}}_{m_n}\right)^2}
&= 1.072 \in [0.916, 1.092].
\end{align*}
Therefore, we accept the existence of strong dependence and conclude $S([0.336,0.853]) = 1$.

To distinguish between full and strong dependence, we obtain $SE_\text{boot}(m_n) = 0.317$, which leads to 
\[
k_{m_n}\frac{SE^2_\text{boot}(m_n)}{H_{k_n,n}^2} = 0.861 < \chi^2_{.95, 1999}/1999 = 1.053.
\]
So we fail to reject the hypothesis of full dependence, i.e. $H^{(2)}_0: S(\{0.482\}) = 1$.
Here even if we consider the rejection rate using the criterion in \eqref{eq:full_test}, only 3.35\% of the 2000 bootstrap trials rejects $H^{(2)}_0$.
Hence, we conclude that the absolute returns of CVX and XOM show full asymptotic dependence.

\subsubsection{CVX vs AAPL}
\begin{figure}
\centering
\includegraphics[width=\textwidth]{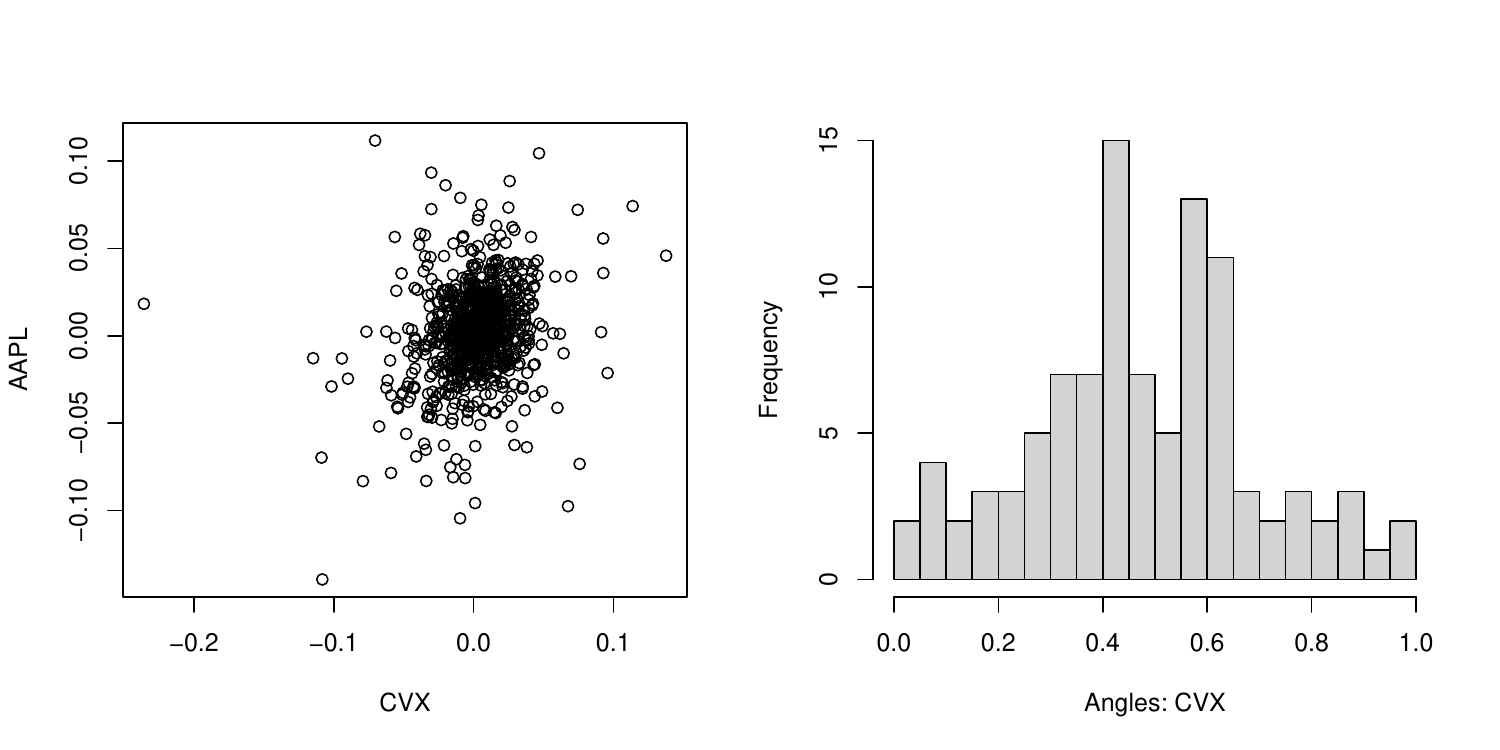}
\caption{CVX vs AAPL. Left: Scatter plot of CVX and AAPL returns. Right: Histogram of angles (absolute returns of CVX) with $k_n=100$.}\label{fig:CVX_AAPL}
\end{figure}

Next, we inspect the dependence structure between absolute returns of CVX and AAPL. Based on analyses using Hill plot (not shown), we choose $k_n=100$ and estimate the marginal tail indices 
$\hat{\alpha} = 1/0.294 = 3.401$.  
We give the scatter plot and the histogram of $\vert x_1\vert/(\vert x_1\vert+\vert x_2\vert)$ in the left and right panels Figure~\ref{fig:CVX_AAPL}, respectively. 

We set $\lambda = 4$ and obtain the estimated $[\hat{a}, \hat{b}] = [0.182, 0.928]$ ($\lambda>4$ gives too wide an interval).
When testing 
\[
H^{(1)}_0: \, S([0.182, 0.928]) = 1\qquad\text{vs}\qquad H^{(1)}_a:\, S([0.182, 0.928]) < 1,
\]
we compute $\widehat{D}_{m_n}^*$ for each of the $2{,}000$ resamples and only 3.5\% of them rejects $H^{(1)}_0$.
Note that the estimated support $[\hat{a}, \hat{b}]$ is already quite wide, 
then the low rejection rate can be a result of having too wide a support of $S$.
So we 
next generate two sets of 2{,}000 bootstrap resamples with $m_n = 200$ and $k_{m_n} = 20$ to test
\[
H^{(3)}_0: \, S([0.182, 0.928]) = 1\qquad\text{vs}\qquad H^{(3)}_a:\, S([0,1]) = 1.
\]
This gives a test statistic 
\begin{align*}
\frac{\frac{1}{B-1}\sum_{t=1}^B \left(T_{m_n}^{(t)}-\bar{T}_{m_n}\right)^2}{\frac{1}{B-1}\sum_{t=1}^B \left(\widetilde{T}_{m_n}^{(t)}-\bar{\widetilde{T}}_{m_n}\right)^2}
&= 0.814 \notin [0.916, 1.092],
\end{align*}
indicating
the existence of weak dependence. Hence, we end up with the conclusion that considering the absolute returns of CVX and AAPL, 
the support of the angular measure is likely to be $[0,1]$.

\section{Supplementary Material}
This supplement contains detailed proofs on the main theorems. Section \ref{sec:pf_strong} proves Theorem~\ref{thm:test_strong}, and Section \ref{sec:pf_fullstrong} proves asymptotic normality of the test statistic $T_n$ under the null hypothesis of full dependence. Section \ref{subsec:fclt} checks conditions of the functional central limit theorem in Theorem 10.6 of \cite{pollard:1984}.

\subsection{Proof of Theorem 3.1}\label{sec:pf_strong}

We start by showing that for 
intermediate sequence $\{k_n\}$ satisfying \eqref{eq:cond1}, 
\begin{align}
W_n(y):=&\sqrt{k_n}\left(\frac{1}{{k_n}}\sum_{i=1}^n
          \left(1+\frac{d^*\bigl((X_i,Y_i),\mathbb{C}_{a,b}\bigr)}{R_{(k_n)}}\right)\ind_{\left\{\frac{R_i}{b(n/k_n)}>y\right\}} 
          - y^{-\alpha}\right)\nonumber \\
  \Rightarrow &W(y^{-\alpha}),\label{eq:claim1}
\end{align}
in $D(0,\infty)$, where $W(\cdot)$ is a standard Brownian motion.

\sid{We begin by showing that the sequence of processes in the variable}
$y$ satisfy
\begin{align}
\label{e:conv0}
\frac{1}{\sqrt{k_n}}\sum_{i=1}^n
  \frac{d^*\bigl((X_i,Y_i),\mathbb{C}_{a,b}\bigr)}{b(n/k_n)}\ind_{\left\{\frac{R_i}{b(n/k_n)}>y\right\}}\Rightarrow
  0, \quad \text{ in }D(0,\infty),
\end{align}
\sid{and it is here that HRV and assumption $[a,b]\subsetneq [0,1]$ is
used.} 
We have  that
\begin{align*}
&\frac{1}{k_n}\sum_{i=1}^n \frac{d^*\bigl((X_i,Y_i),\mathbb{C}_{a,b}\bigr)}{b(n/k_n)}\ind_{\left\{\frac{R_i}{b(n/k_n)}>y\right\}}\\
& =
\frac{b_0(n/k_n)}{b(n/k_n)} \frac{1}{k_n}\sum_{i=1}^n \frac{d^*\bigl((X_i,Y_i),\mathbb{C}_{a,b}\bigr)}{b_0(n/k_n)}\ind_{\left\{\frac{R_i}{b(n/k_n)}>y\right\}}\\
 &\le \frac{b_0(n/k_n)}{b(n/k_n)} \frac{1}{k_n}\sum_{i=1}^n \frac{d^*\bigl((X_i,Y_i),\mathbb{C}_{a,b}\bigr)}{b_0(n/k_n)}\left(\ind_{\left\{\frac{R_i}{b(n/k_n)}>y, d^*\bigl((X_i,Y_i),\mathbb{C}_{a,b}\bigr)> b_0(n/k_n)\epsilon\right\}}\right.\\
 &\left. \qquad +\ind_{\left\{\frac{R_i}{b(n/k_n)}>y, d^*\bigl((X_i,Y_i),\mathbb{C}_{a,b}\bigr)\le  b_0(n/k_n)\epsilon\right\}}\right)\\
 &\le \frac{b_0(n/k_n)}{b(n/k_n)}\left(\frac{1}{k_n}\sum_{i=1}^n
   \frac{d^*\bigl((X_i,Y_i),\mathbb{C}_{a,b}\bigr)}{b_0(n/k_n)}\ind_{\left\{\frac{d^*\bigl((X_i,Y_i),\mathbb{C}_{a,b}\bigr)}{
   b_0(n/k_n)}>\epsilon\right\}}+\frac{\epsilon}{k_n}\sum_{i=1}^n\ind_{\left\{\frac{R_i}{b(n/k_n)}>y\right\}}\right)\\
&=A+B.
\end{align*}

\sid{To handle $B$ observe for each fixed $y>0$, that the monotone function
in $y$,
$$\frac{1}{k_n}\sum_{i=1}^n\ind_{\left\{\frac{R_i}{b(n/k_n)}>y\right\}}
\Rightarrow
y^{-\alpha_0}
$$ using, for example \cite[Theorem 5.3(ii), p. 139]{resnickbook:2007}.
Therefore,
when $k_n$ satisfies
\eqref{eq:cond1}, we have $\sqrt k_n B \Rightarrow 0 $ in
$D(0,\infty)$.}

For $A$ we claim
$$\frac{1}{k_n}\sum_{i=1}^n
   \frac{d^*\bigl((X_i,Y_i),\mathbb{C}_{a,b}\bigr)}{b_0(n/k_n)}\ind_{\left\{\frac{d^*\bigl((X_i,Y_i),\mathbb{C}_{a,b}\bigr)}{
   b_0(n/k_n)}>\epsilon\right\}} =O_p(1),
$$
since for any $M>0$ 
\begin{align*}
  \PP \Biggl[ \frac{1}{k_n}\sum_{i=1}^n&
   \frac{d^*\bigl((X_i,Y_i),\mathbb{C}_{a,b}\bigr)}{b_0(n/k_n)}\ind_{\left\{\frac{d^*\bigl((X_i,Y_i),\mathbb{C}_{a,b}\bigr)}{
                                                             b_0(n/k_n)}>\epsilon\right\}} >M \Biggr]\\
  \leq &\frac 1M \frac{n}{k_n} \EE \left(
      \frac{   d^*\bigl((X_1,Y_1),\mathbb{C}_{a,b}\bigr)}{b_0(n/k_n)}
         \ind_{\left\{\frac{d^*\bigl((X_1,Y_1),\mathbb{C}_{a,b}\bigr)}{
                                                             b_0(n/k_n)}>\epsilon\right\}}\right)\\
 \intertext{and because $\alpha_0>1$ we may apply Karamata's theorem
  on integration to get convergence, as $n\to\infty$ to}
  \to &\frac 1M \int_\epsilon^\infty x\nu_{\alpha_0} (dx)<\infty.
\end{align*}
Therefore $\sqrt k_n  A \Rightarrow 0$ in $D(0,\infty)$. This proves \eqref{eq:claim1}.

Since $R_{(k_n)}/b(n/k_n)\convp 1$
(eg. \cite[p. 82]{resnickbook:2007}), we also have
\[
\frac{1}{\sqrt{k_n}}\sum_{i=1}^n
\frac{d^*\bigl((X_i,Y_i),\mathbb{C}_{a,b}\bigr)}{R_{(k_n)}}\ind_{\left\{\frac{R_i}{b(n/k_n)}>y\right\}}\Rightarrow
0,\quad  \text{ in }D(0,\infty).
\]

So to prove \eqref{eq:claim1} it remains to verify
\begin{equation}\label{e:tailE}
\sqrt{k_n}\left(\frac{1}{{k_n}}\sum_{i=1}^n \ind_{\left\{\frac{R_i}{b(n/k_n)}>y\right\}}
  - y^{-\alpha}\right)\Rightarrow W(y^{-\alpha}),
\end{equation}
in $D(0,\infty)$. The regular variation of $P[R_1>x]$ implies
(\cite[Theorem 9.1, p. 292]{resnickbook:2007} or
\cite{dehaan:ferreira:2006})  that
$$
\sqrt{k_n}\left(\frac{1}{{k_n}}\sum_{i=1}^n \ind_{\left\{\frac{R_i}{b(n/k_n)}>y\right\}}
  - \frac{n}{k_n}\PP [R_1/b(n/k) >y]\right)\Rightarrow W(y^{-\alpha}),
$$
in $D(0,\infty)$ and  the 2RV assumption \sid{in \eqref{e:2rv}
  marginalized to the distribution of $R_1$ and the choice of $k_n$
in \eqref{eq:cond1} imply}
\[
\sqrt{k_n}\left(\frac{n}{k_n}\PP\left(\frac{R_1}{b(n/k_n)}>y\right)-y^{-\alpha}\right) \to 0,
\]
\sid{locally uniformly in $y$}.
This gives \eqref{e:tailE} and completes the proof of \eqref{eq:claim1}.

Apply the composition map $(x(t),c)\mapsto x(ct)$
\sid{from $D(0,\infty)\times (0,\infty)\mapsto D(0,\infty)$ 
to \eqref{eq:claim1} in the form
$$ \Bigl(  W_n(y), \frac{R_{(k_n)}}{b(n/k)} \Bigr)\mapsto
  W_n(\frac{R_{(k_n)}}{b(n/k)} y) \Rightarrow W(y^{-\alpha})
    $$
} to get
\begin{align}
\label{eq:dstar}
\sqrt{k_n}\Biggl( &\frac{1}{{k_n}}\sum_{i=1}^n \left(1+\frac{d^*\bigl((X_i,Y_i),\mathbb{C}_{a,b}\bigr)}{R_{(k_n)}}\right)\ind_{\left\{\frac{R_i}{R_{(k_n)}}>y\right\}}
- \sid{\bigl( \frac{R_{(k_n)}}{b(n/k)}     y
                   \bigr)^{-\alpha} }  \Biggr)  \nonumber\\
  &\Rightarrow W(y^{-\alpha}). 
\end{align}
\sid{Couple this with a Vervaat inversion of \eqref{e:tailE}
  (\cite[p. 357]{dehaan:ferreira:2006} or
  \cite[p. 57]{resnickbook:2007}). The inversion 
yields in $D(0,\infty)$ 
\begin{equation}\label{e:Wim}
  \sqrt{k_n}\Biggl(
\Bigl(  \frac{R_{([kt])}}{b(n/k_n)} \Bigr)^{-\alpha} - t\Biggr)
\Rightarrow -W(t),\quad \text{ in }D(0,\infty),
\end{equation}
  and the convergence is joint with the one in \eqref{eq:dstar}.
  Combining \eqref{eq:dstar} with \eqref{e:Wim} gives
\begin{align}
\sqrt{k_n}\Bigl( &\frac{1}{{k_n}}\sum_{i=1}^n \left(1+\frac{d^*\bigl((X_i,Y_i),\mathbb{C}_{a,b}\bigr)}{R_{(k_n)}}\right)\ind_{\left\{\frac{R_i}{R_{(k_n)}}>y\right\}}
-   y^{-\alpha} \Bigr) \nonumber\\
                 &\Rightarrow W(y^{-\alpha})-y^{-\alpha}W(1). \label{e:finally}
\end{align}
  }

Finally apply the mapping
\[
x\mapsto \int_1^\infty \frac{x(s)}{s}\dd s
\]
to \eqref{eq:dstar} \sid{using justifications similar to \cite[Section
9.1]{resnickbook:2007}.} This yields the asymptotic normality of
$D^*_n$ under the null hypothesis in $H^{(1)}_0$,
\begin{align*}
\sqrt{k_n}&\left(\frac{1}{{k_n}}\sum_{i=1}^n \left(1+\frac{d^*\bigl((X^*_i,Y^*_i),\mathbb{C}_{a,b}\bigr)}{R_{(k_n)}}\right)\log\frac{R_{(i)}}{R_{(k_n)}}
- \frac{1}{\alpha}\right)\\
&\Rightarrow \frac{1}{\alpha}\left(\int_0^1 \frac{W(s)}{s}\dd
                              s-W(1)\right)\stackrel{d}{=}\frac{1}{\alpha}W(1)\sim N(0,1/\alpha^2).
\end{align*}

\subsection{Proof of Theorem 4.1}\label{sec:pf_fullstrong}
 The proof proceeds in a series of steps.
  \medskip
  \begin{enumerate}
  \item
    To prove the convergence in Eq.(23) of main document under $H^{(2)}_0$, we first show that in $D(0,\infty)$,
\begin{align}
\label{eq:step1a}
\tilde W_n(y):=\frac{\sqrt{k_n}}{\theta_0}\left(\frac{1}{k_n}\sum_{i=1}^n \Theta_i\epsilon_{R_i/b(n/k_n)}(y,\infty) - \theta_0 y^{-\alpha} \right)
\Rightarrow  W(y^{-\alpha}).
\end{align}
The LHS of \eqref{eq:step1a} \sid{can be decomposed as}
\begin{align}
\label{eq:step1a_LHS}
\frac{1}{\theta_0}\frac{1}{\sqrt{k_n}}\sum_{i=1}^n&
  \left(\Theta_i-\theta_0\right)\epsilon_{R_i/b(n/k_n)}(y,\infty)
                                  \nonumber \\
&+ \sqrt{k_n}\left(\frac{1}{k_n}\sum_{i=1}^n \epsilon_{R_i/b(n/k_n)}(y,\infty) - y^{-\alpha}\right)=:B1+B2.
\end{align}
From full dependence, we have $\mathbb{C}_{a,b} \equiv
\{(x,y)\in \mathbb{R}_+^2: y=(1/\theta_0-1) x\}$. \sid{Remember
  $\Theta_i=X_i/R_i=X_i/(X_i+Y_i)$ and then
$\vert B1\vert$ is bounded by}
\begin{align}
\frac{1}{\theta_0}\frac{1}{\sqrt{k_n}}&\sum_{i=1}^n
\Bigl \vert\frac{X_i}{R_i}-\theta_0\Bigr\vert\epsilon_{R_i/b(n/k_n)}(y,\infty)\nonumber\\
  =&\frac{1}{\theta_0}\frac{1}{\sqrt{k_n}}
                                        \sum_{i=1}^n \theta_0 
\frac{\vert Y_i-X_i(\theta_0^{-1}-1) \vert }{R_i} \ind_{\{R_i>b(n/k_n)y\}}       \nonumber\\
  =& \frac{1}{\sqrt{k_n}}\sum_{i=1}^n\frac{d^*((X_i,Y_i),
         \mathbb{C}_{a,b})}{b(n/k_n)y}\ind_{\{R_i>b(n/k_n)y\}} 
\Rightarrow 0\label{eq:conv_dist}
\end{align}
in $D(0,\infty)$, \sid{since   \eqref{e:conv0} is still applicable.}
\sid{As in \eqref{e:tailE},} the second term \sid{$B2$} in \eqref{eq:step1a_LHS} converges weakly \sid{in
$D(0,\infty)$} to
$W(y^{-\alpha})$ under the 2RV condition for $\PP[R_1>x]$, thus
completing the proof of \eqref{eq:step1a}.

Combine \eqref{eq:step1a} with 
\[
\frac{R_{(k_n)}}{b(n/k_n)}\convp 1,
\]
\sid{to get joint convergence in $D(0,\infty)\times \RR_+$.} Applying
the 
composition map $(x(t),c)\mapsto x(ct)$ from
\sid{$D(0,\infty)\times (0,\infty)\mapsto D(0,\infty)$}, we get in $D(0,\infty)$,
\begin{align}
\label{eq:step1b}
\frac{\sqrt{k_n}}{\theta_0}\left(\frac{1}{k_n}\sum_{i=1}^n \Theta_i\epsilon_{R_i/R_{(k_n)}}(y,\infty) 
- \theta_0 \left(\frac{R_{(k_n)}}{b(n/k_n)} y\right)^{-\alpha} \right)
\Rightarrow  W(y^{-\alpha}).
\end{align}
Apply Vervaat inversion again as in \eqref{e:Wim} and
\eqref{e:finally}, we again conclude
\begin{equation}\label{e:center1st}
\sqrt{k_n} \left(\frac{1}{k_n \theta_0}\sum_{i=1}^n \Theta_i\epsilon_{R_i/R_{(k_n)}}(y,\infty) 
- y^{-\alpha }  \right)
\Rightarrow  W(y^{-\alpha}) -y^{-\alpha}W(1).
\end{equation}


\item 
Next, we need to justify application of the map
\[
x\mapsto \int_1^\infty \frac{x(s)}{s}\dd s
\]
to \eqref{e:center1st}, which, if justified, leads to
\begin{equation}\label{e:step2}
\sqrt{k_n}
\left(\frac{1}{k_n \theta_0}\sum_{i=1}^{k_n}
  \Theta^*_i\log \frac{R_{(i)}}{R_{(k_n)}}
  -\frac{1}{\alpha} \right)
\Rightarrow
\int_1^\infty \frac{W(y^{-\alpha})}{y}\dd y -\frac{1}{\alpha} W(1)\sim
\frac{1}{\alpha}N(0,1).
 \end{equation}

The proof is similar to the one given in Proposition~9.1 of \cite{resnickbook:2007},
and we defer details to Section~\ref{subsec:mapping_pf}.

\medskip

\item From \eqref{e:step2}, 
\begin{equation}
\label{eq:consistency}
\frac{1}{k_n}\sum_{i=1}^{k_n} \Theta^*_i\log \frac{R_{(i)}}{R_{(k_n)}}
\convp \frac{\theta_0}{\alpha},
\end{equation}
\sid{which suggests comparing}
\begin{align}\label{e:Comp}
\sqrt{k_n}\Bigl(T_n &- \frac{1}{\theta_0}\frac{1}{k_n}\sum_{i=1}^{k_n}
            \Theta^*_i\log \frac{R_{(i)}}{R_{(k_n)}} \Bigr)  \\
=& \sqrt{k_n}\Bigl(\frac{1}{\frac{1}{k_n}\sum_{i=1}^{k_n} \Theta^*_i}
                                                                                -
   \frac{1}{\theta_0}\Bigr)\frac{1}{k_n}\sum_{i=1}^{k_n}
   \Theta^*_i\log \frac{R_{(i)}}{R_{(k_n)}}  \nonumber\\
  \intertext{and applying \eqref{eq:consistency}, this is}
  =& \sqrt{k_n}\Bigl(\frac{1}{\frac{1}{k_n}\sum_{i=1}^{k_n}
     \Theta^*_i}      -
     \frac{1}{\theta_0}\Bigr) O_p(1)=
     \sqrt{k_n}\Biggl(\frac{\theta_0-\frac{1}{k_n} \sum_{i=1}^{k_n}
     \Theta^*_i     }{\frac{1}{k_n}\sum_{i=1}^{k_n}
     \Theta^*_i \theta_0}                                                                                 
     \Biggr) O_p(1)
     .\label{e:oneMore}
\end{align}
Now
$$\frac{1}{k_n} \sum_{i=1}^n \epsilon_{(\Theta_i, R_i/R_{(k)})}
\Rightarrow S\times \nu_\alpha=\epsilon_{\theta_0} \times \nu_\alpha
$$
in $\mathbb{M}([0,1]\times \RR_+\setminus \{0\})$
(eg. \cite[p. 180]{resnickbook:2007})
and so
\begin{align*}
\int_{[0,1]\times (1,\infty)} \theta &\frac{1}{k_n} \sum_{i=1}^n
\epsilon_{(\Theta_i, R_i/R_{(k)})} (d\theta,dr)=\frac{1}{k_n}
  \sum_{i=1}^n \Theta_i \ind_{[R_i>R_{(k_n)}]}\\
  =&\frac{1}{k_n} \sum_{i=1}^{k_n} \Theta_i^*
  \Rightarrow 
\int_{[0,1]\times (1,\infty)} \theta \epsilon_{\theta_0}\times
  \nu_\alpha(d\theta,dr)= \theta_0.
\end{align*}
Thus the denominator of \eqref{e:oneMore} is also $O_p(1)$.
A successful comparison in \eqref{e:Comp} has the difference
converging to $0$ and so it remains to show
\begin{equation}\label{e:please0}
  \frac{1}{\sqrt{k_n}}\Bigl \vert k_n \theta_0- \sum_{i=1}^{k_n}
  \Theta^*_i   \Bigr \vert \Rightarrow 0.\end{equation}
  
Since under $H^{(2)}_0$,
\begin{align*}
\frac{1}{\sqrt{k_n}}\Bigl \vert\sum_{i=1}^{k_n}(\Theta_i^*&- \theta_0)\Bigr\vert
= \frac{1}{\sqrt{k_n}}\left\vert\sum_{i=1}^{n}(\Theta_i-
                                                           \theta_0)\ind_{\{R_i\ge
                                                           R_{(k_n)}\}}\right\vert
  \\
&\le \frac{1}{\sqrt{k_n}}\sum_{i=1}^n                                                                     \frac{d^*((X_i,Y_i),\mathbb{C}_{a,b})}{R_{(k_n)}}\ind_{\{R_i\ge R_{(k_n)}\}}
  \convp 0,
\end{align*}
which can be seen as in the proof of \eqref{e:conv0} by replacement
of $R_{(k_n)}$ by 
$b(n/k_n)$ at the cost of $1\pm \epsilon$ for any
$\epsilon>0$ with high probability.
This confirms \eqref{e:please0} and thus proves convergence to 0 in
\eqref{e:Comp}. In turn, this coupled with \eqref{e:step2} proves the theorem.
\end{enumerate}

\subsubsection{Details for Step 2 in the proof of Theorem~4.1}\label{subsec:mapping_pf}
The proof of \eqref{e:step2} requires justifying the application of the mapping
\[
x\mapsto \int_1^\infty x(s)\frac{\dd s}{s},
\]
and applying this mapping to \eqref{eq:step1b} 
leads to 
\begin{align}\label{eq:step2a}
\frac{\sqrt{k_n}}{\theta_0}\left(\frac{1}{k_n}\sum_{i=1}^{k_n} \Theta^*_i\log \frac{R_{(i)}}{R_{(k_n)}} - \frac{\theta_0}{\alpha}\left(\frac{R_{(k_n)}}{b(n/k_n)}\right)^{-\alpha}\right)
\Rightarrow \int_1^\infty \frac{W(y^{-\alpha})}{y}\dd y.
\end{align}
For $M$ large, applying the map
\[
x\mapsto \int_1^M \frac{x(s)}{s}\dd s
\]
to \eqref{eq:step1b} gives
\begin{align*}
\frac{\sqrt{k_n}}{\theta_0}&\left(\int_1^M\frac{1}{k_n}\sum_{i=1}^n \Theta_i\epsilon_{R_i/R_{(k_n)}}(s,\infty)\frac{\dd s}{s} 
- \left(\frac{R_{(k_n)}}{b(n/k_n)}\right)^{-\alpha}\theta_0 \int_1^M s^{-\alpha-1}\dd s\right)\\
&\Rightarrow \int_1^M W(s^{-\alpha})\frac{\dd s}{s}.
\end{align*}
As $M\to\infty$,
\[
\int_1^M W(s^{-\alpha})\frac{\dd s}{s}
\to \int_1^\infty \frac{W(s^{-\alpha})}{s}\dd s.
\]
Hence, it remains to verify that for any $\delta>0$,
\begin{align}
\label{eq:step2_map}
\lim_{M\to\infty}\limsup_{n\to\infty}\, &\PP\left(\sqrt{k_n}\left\vert
\int_M^\infty\frac{1}{k_n}\sum_{i=1}^n \Theta_i \epsilon_{R_i/R_{(k_n)}}(s,\infty)\frac{\dd s}{s}\right.\right.\nonumber\\
&\left.\left.\qquad-\theta_0 \left(\frac{R_{(k_n)}}{b(n/k_n)}\right)^{-\alpha}\int_M^\infty s^{-\alpha-1}\dd s
\right\vert>\delta\right)=0
\end{align}
Rewrite the probability in \eqref{eq:step2_map} as
\begin{align*}
\PP&\left(\sqrt{k_n}\left\vert
\int_M^\infty \left(\frac{1}{k_n}\sum_{i=1}^n \Theta_i \epsilon_{R_i/R_{(k_n)}}(s,\infty)
- \theta_0\left(\frac{R_{(k_n)}}{b(n/k_n)}s\right)^{-\alpha}\right)\frac{\dd s}{s}
\right\vert>\delta\right)\\
&\le \PP\left(\sqrt{k_n}
\int_M^\infty\left\vert\frac{1}{k_n}\sum_{i=1}^n \Theta_i \epsilon_{R_i/R_{(k_n)}}(s,\infty)
- \theta_0\left(\frac{R_{(k_n)}}{b(n/k_n)}s\right)^{-\alpha}\right\vert\frac{\dd s}{s}
>\delta\right)\\
&= \PP\left(\sqrt{k_n}
\int_{M R_{(k_n)}/b(n/k_n)}^\infty\left\vert\frac{1}{k_n}\sum_{i=1}^n \Theta_i \epsilon_{R_i/b(n/k_n)}(s,\infty)
- \theta_0 s^{-\alpha}\right\vert\frac{\dd s}{s}
>\delta\right)\\
&\le \PP\left(\sqrt{k_n}
\int_{M (1-\eta)}^\infty\left\vert\frac{1}{k_n}\sum_{i=1}^n \Theta_i \epsilon_{R_i/b(n/k_n)}(s,\infty)
- \theta_0 s^{-\alpha}\right\vert\frac{\dd s}{s}
>\delta\right)\\
&\qquad\qquad + \PP\left(\left\vert \frac{R_{(k_n)}}{b(n/k_n)}-1\right\vert>\eta\right),
\end{align*}
for $\eta>0$. Since $R_{(k_n)}/b(n/k_n)\convp 1$, it suffices to consider
\begin{align*}
\PP&\left(\sqrt{k_n}
\int_{M (1-\eta)}^\infty\left\vert\frac{1}{k_n}\sum_{i=1}^n \Theta_i \epsilon_{R_i/b(n/k_n)}(s,\infty)
- \theta_0 s^{-\alpha}\right\vert\frac{\dd s}{s}
>\delta\right)\\
& \le \frac{k_n}{\delta^2}\EE\left[\int_{M
                   (1-\eta)}^\infty\left(\frac{1}{k_n}\sum_{i=1}^n
                   \Theta_i \epsilon_{R_i/b(n/k_n)}(s,\infty) 
                   -\theta_0 s^{-\alpha}\right)^2\frac{\dd
                   s}{s}\right].
\end{align*}
Write what is inside the square by centering the random term to get
  \begin{align*}
    \frac{1}{k_n}&\sum_{i=1}^n
  \Theta_i \epsilon_{R_i/b(n/k_n)}(s,\infty) -\frac{n}{k_n}
  \EE\left(\Theta_1\ind_{\{R_1>b(n/k_n)s\}}\right)
\\
    +&
  \frac{n}{k_n}
  \EE\left(\Theta_1\ind_{\{R_1>b(n/k_n)s\}}\right)- \theta_0
  s^{-\alpha}
 \end{align*}
 and we
  get the probability in \eqref{eq:step2_map} bounded by
  \begin{align*}
&\le \frac{k_n}{\delta^2} \int_{M (1-\eta)}^\infty \frac{n}{k_n^2}\text{Var}\left(\Theta_1\ind_{\{R_1>b(n/k_n)s\}}\right)\frac{\dd s }{s}\\
&\quad+ \frac{k_n}{\delta^2} \int_{M (1-\eta)}^\infty \left(\frac{n}{k_n}\EE\left(\Theta_1\ind_{\{R_1>b(n/k_n)s\}}\right)- \theta_0 s^{-\alpha}\right)^2\frac{\dd s}{s}\\
&=: I_n + II_n
\end{align*}
Since $\Theta_1\le 1$ a.s., 
the term $I_n$ is bounded by 
\[
  \frac{1}{\delta^2}\int_{M (1-\eta)}^\infty \frac{n}{k_n}
  \EE\left(\Theta_1\ind_{\{R_1>b(n/k_n)s\}}\right)^{\sid{2}}\frac{\dd s}{s}
\le \frac{1}{\delta^2}\int_{M (1-\eta)}^\infty \frac{n}{k_n} \PP(\{R_1>b(n/k_n)s)\frac{\dd s}{s}.
\]
By Karamata's theorem, the right side converges as $n\to\infty$ to
$$\frac{1}{\delta^2}\int_{M (1-\eta)}^\infty s^{-\alpha-1}\dd s = \frac{1}{\alpha\delta^2} (M(1-\eta))^{-\alpha}
\stackrel{M\to\infty}{\longrightarrow}0.
$$

For $II_n$, with $v(t) = \EE(\Theta_1 \ind_{\{R_1>t\}})$, we notice that
\begin{align*}
&\frac{\theta_0^2}{ A^2(n/k_n)}\left(\frac{n}{k_n\theta_0}\EE\left(\Theta_1\ind_{\{R_1>b(n/k_n)s\}}\right)-  s^{-\alpha}\right)^2\\
&= \frac{\theta_0^2}{A^2(n/k_n)}\left(\frac{n}{k_n\theta_0}\EE\left(\Theta_1\ind_{\{R_1>b(n/k_n)s\}}\right)-\frac{v(b(n/k_n)s)}{v(b(n/k_n))} +  \frac{v(b(n/k_n)s)}{v(b(n/k_n))}- s^{-\alpha}\right)^2\\
&\le \left\vert\frac{\theta_0}{A(n/k_n)}\left(\frac{n}{k_n\theta_0}\EE\left(\Theta_1\ind_{\{R_1>b(n/k_n)s\}}\right)-\frac{v(b(n/k_n)s)}{v(b(n/k_n))}\right)\right\vert^2\\
&\quad + \left\vert\frac{\theta_0}{A(n/k_n)}\left(\frac{v(b(n/k_n)s)}{v(b(n/k_n))}- s^{-\alpha}\right)\right\vert^2\\
&\quad + 2\left\vert\frac{\theta_0}{A(n/k_n)}\left(\frac{n}{k_n\theta_0}\EE\left(\Theta_1\ind_{\{R_1>b(n/k_n)s\}}\right)-\frac{v(b(n/k_n)s)}{v(b(n/k_n))}\right)\right\vert\\
&\qquad\times\left\vert\frac{\theta_0}{A(n/k_n)}\left(\frac{v(b(n/k_n)s)}{v(b(n/k_n))}- s^{-\alpha}\right)\right\vert.
\end{align*}
By \cite[Theorem~2.3.9]{dehaan:ferreira:2006}, for any $\epsilon>0$, $\delta\in (0,\alpha(1+\rho))$, 
there exists $A_0(n/k_n)\sim A(n/k_n)$ as $n\to\infty$ and $n_0\equiv n_0(\epsilon,\delta)$ such that for all $b(n/k_n), b(n/k_n)s\ge n_0$,
\begin{align*}
\left\vert\frac{1}{A_0(n/k_n)}\left(\frac{v(b(n/k_n)s)}{v(b(n/k_n))}- s^{-\alpha}\right)-s^{-\alpha}\frac{1-s^{-\alpha\rho}}{\alpha\rho}\right\vert
\le \epsilon s^{-\alpha(1+\rho)}\max\{s^{-\delta},s^{\delta}\},
\end{align*}
so that
\begin{align*}
&\left\vert\frac{1}{A_0(n/k_n)}\left(\frac{v(b(n/k_n)s)}{v(b(n/k_n))}- s^{-\alpha}\right)\right\vert\\
&\le s^{-\alpha}\left(\left\vert\frac{1-s^{-\alpha\rho}}{\alpha\rho}\right\vert+\epsilon s^{-\alpha\rho}\max\{s^{-\delta},s^{\delta}\}\right).
\end{align*}
Hence, 
\begin{align*}
\int_{M(1-\eta)}^\infty& \left\vert\frac{\theta_0}{A_0(n/k_n)}\left(\frac{v(b(n/k_n)s)}{v(b(n/k_n))}- s^{-\alpha}\right)\right\vert^2\frac{\dd s}{s}\\
&\le \int_{M(1-\eta)}^\infty s^{-2\alpha}\left(\left\vert\frac{1-s^{-\alpha\rho}}{\alpha\rho}\right\vert+\epsilon s^{-\alpha\rho}\max\{s^{-\delta},s^{\delta}\}\right)^2\frac{\dd s}{s}<\infty,
\end{align*}
which further implies
\begin{align*}
\int_{M(1-\eta)}^\infty& \left\vert\frac{\theta_0}{A(n/k_n)}\left(\frac{v(b(n/k_n)s)}{v(b(n/k_n))}- s^{-\alpha}\right)\right\vert^2\frac{\dd s}{s}<\infty,
\end{align*}
as $A_0(n/k_n)\sim A(n/k_n)$. 
In addition, since
\begin{align*}
\frac{\theta_0}{A(n/k_n)}&\left(\frac{n}{k_n\theta_0}\EE\left(\Theta_1\ind_{\{R_1>b(n/k_n)s\}}\right)-\frac{v(b(n/k_n)s)}{v(b(n/k_n))}\right)\\
&= \frac{v(b(n/k_n)s)}{v(b(n/k_n))\theta_0}\frac{1}{A(n/k_n)}\left(\frac{n}{k_n}v(b(n/k_n))-\theta_0\right),
\end{align*}
then 
\begin{align*}
\int_{M(1-\eta)}^\infty &\left\vert\frac{\theta_0}{A(n/k_n)}\left(\frac{n}{k_n\theta_0}\EE\left(\Theta_1\ind_{\{R_1>b(n/k_n)s\}}\right)-\frac{v(b(n/k_n)s)}{v(b(n/k_n))}\right)\right\vert^2\frac{\dd s}{s}\\
& = \left\vert\frac{1}{A(n/k_n)}\left(\frac{n}{k_n}v(b(n/k_n))-\theta_0\right)\right\vert^2 \int_{M(1-\eta)}^\infty \frac{v^2(b(n/k_n)s)}{\theta_0^2 v^2(b(n/k_n))}\frac{\dd s}{s}\\
&\to \left(\frac{1}{\theta_0}\int\int_{(1,\infty)\times [0,1]}\theta\chi(\dd r, \dd\theta)\right)^2 \frac{(M(1-\eta))^{-2\alpha}}{2\alpha}<\infty.
\end{align*}
Since under the condition in \eqref{eq:cond1}, the intermediate sequence $\{k_n\}$ also satisfies
$\sqrt{k_n}A(n/k_n)\to 0$ as $n\to\infty$. We then see that
\begin{align*}
k_n& \int_{M(1-\eta)}^\infty\left\vert\left(\frac{n}{k_n}\EE\left(\Theta_1\ind_{\{R_1>b(n/k_n)s\}}\right)-\theta_0\frac{v(b(n/k_n)s)}{v(b(n/k_n))}\right)\right\vert^2\frac{\dd s}{s}\\
=&(\sqrt{k_n}A(n/k_n))^2 \\
&\quad\times\int_{M(1-\eta)}^\infty\left\vert\frac{\theta_0}{A(n/k_n)}\left(\frac{n}{k_n\theta_0}\EE\left(\Theta_1\ind_{\{R_1>b(n/k_n)s\}}\right)-\frac{v(b(n/k_n)s)}{v(b(n/k_n))}\right)\right\vert^2\frac{\dd s}{s}\to 0,\\
\intertext{and similarly,}
k_n&\int_{M(1-\eta)}^\infty \left\vert {\theta_0}\left(\frac{v(b(n/k_n)s)}{v(b(n/k_n))}- s^{-\alpha}\right)\right\vert^2 \frac{\dd s}{s}\to 0.
\end{align*}
So we conclude that $II_n\to 0$. This justifies \eqref{eq:step2_map}, thus completing the proof of \eqref{eq:step2a}.


Recall \eqref{e:center1st}, and applying the mapping
\[
(x,y)\mapsto \left(\int_1^\infty x(s)\frac{\dd s}{s}, y\right)
\]
to \eqref{e:center1st} gives
\begin{align}
\label{eq:step3}
\frac{\sqrt{k_n}}{\theta_0}\left(\frac{1}{k_n}\sum_{i=1}^{k_n} \Theta^*_i\log \frac{R_{(i)}}{R_{(k_n)}} - \frac{\theta_0}{\alpha}\right)
\Rightarrow \int_1^\infty \frac{W(y^{-\alpha})}{y}\dd y-\frac{1}{\alpha}W(1).
\end{align}

\subsection{Check conditions of the functional central limit theorem}\label{subsec:fclt}
We now present the functional central limit theorem given in \cite[Theorem~10.6]{pollard:1984}.
\begin{Theorem}
\label{thm:fclt}
Consider the triangular array $\{f_{n,i}(t): t\in T\}$ with envelop function $F_{n,i}$, independent within each row. Suppose also that $\{f_{n,i}\}$ satisfy
\begin{enumerate}
\item[(i)] $\{f_{n,i}\}$ are manageable;
\item[(ii)] For $X_n(t)= \sum_i (f_{n,i}(t)-\EE(f_{n,i}(t)))$,
$H(s,t)=\lim_{n\to\infty} \EE[X_n(t)X_n(s)]$ exists for every $s,t\in T$;
\item[(iii)] The envelope function satisfies $\limsup_{n\to\infty}\EE(F^2_{n,i})<\infty$, and
\[
\sum_i\EE\left(F^2_{n,i}\ind_{\{F_{n,i}>\eta\}}\right)\to 0,
\]
for each $\eta>0$;
\item[(iv)] Let $\rho_n(s,t)=\left(\sum_i \EE\left(f_{n,i}(t)-f_{n,i}(s)\right)^2\right)^{1/2}$, then 
the limit $\rho(s,t)=\lim_{n\to\infty}\rho_n(s,t)$ is well-defined, and for deterministic sequences $\{s_n\}$, $\{t_n\}$,
if $\rho(s_n,t_n)\to 0$, then $\rho_n(s_n,t_n)\to 0$.
\end{enumerate}
Then $X_n$ converges to a Gaussian process with zero mean and covariance given by $H$.
\end{Theorem}

To align with the statement in Theorem~\ref{thm:fclt}, we define
\[
f_{n,i}(t) := \frac{(\mu^2+\sigma^2)^{-1/2}}{ k_n}\Theta_i \epsilon_{R_i/b(n/k_n)}(t^{-1/\alpha},\infty),\qquad t\in (0,1],
\]
and the envelope function 
\[
F_{n,i} :=  \frac{(\mu^2+\sigma^2)^{-1/2}}{ k_n} \epsilon_{R_i/b(n/k_n)}(1,\infty).
\]
By Definition~7.9 of \cite{pollard:1984}, we see that $\{f_{n,i}\}$ are manageable.
Also, with $$X_n(t) = \sum_{i=1}^n \left(f_{n,i}(t)-\EE[f_{n,i}(t)]\right),$$
we have 
\begin{align*}
\EE&\left(X_n(t)X_n(s)\right)= \frac{(\mu^2+\sigma^2)^{-1}}{ k_n}\sum_{i=1}^n \EE\left(\Theta_i^2 \epsilon_{R_i/b(n/k_n)}\bigl((t\wedge s)^{-1/\alpha},\infty\bigr)\right)\\
& - \frac{(\mu^2+\sigma^2)^{-1}}{ k_n}\sum_{i=1}^n \EE\left(\Theta_i \epsilon_{R_i/b(n/k_n)}\bigl(t^{-1/\alpha},\infty\bigr)\right)\EE\left(\Theta_i \epsilon_{R_i/b(n/k_n)}\bigl(s^{-1/\alpha},\infty\bigr)\right)\\
&\longrightarrow t\wedge s,\qquad n\to\infty.
\end{align*}
For the envelope function $F_{n,i}$, we have
\begin{align*}
\limsup_{n\to\infty}\sum_{i=1}^n\EE(F^2_{n,i})\to (\mu^2+\sigma^2)^{-1/2},
\end{align*}
and for $\delta>0$,
\begin{align*}
\sum_{i=1}^n\EE(F^{2+\delta}_{n,i})&=(\mu^2+\sigma^2)^{-1/2}\frac{n}{k_n^{2+\delta}}\PP\left(\frac{R_1}{b(n/k_n)}>1\right)\to 0,
\end{align*}
which further implies for each $\eta>0$,
\begin{align*}
\sum_{i=1}^n\EE(F^2_{n,i}\ind_{\{F_{n,i}>\eta\}})\to 0.
\end{align*}
Assume $t_1>t_2$, then for 
\[
\rho_n(t_1,t_2):= \left(\sum_{i=1}^n \EE\left(f_{n,i}(t_1)-f_{n,i}(t_2)\right)^2\right)^{1/2},
\]
we have
\begin{align*}
\rho_n(t_1,t_2) &= (\mu^2+\sigma^2)^{-1/2}\left(\frac{n}{k_n}\EE\left(\Theta_1\epsilon_{R_1/b(n/k_n)}(t_1^{-1/\alpha},t_2^{-1/\alpha}\right)\right)^{1/2}\\
&\longrightarrow (\mu^2+\sigma^2)^{-1/2}(t_1-t_2).
\end{align*}
Therefore, all conditions in Theorem~\ref{thm:fclt} are satisfied, which gives the conclusion that in $D(0,1]$,
\begin{align*}
X_n(t) &= \frac{(\mu^2+\sigma^2)^{-1/2}}{\sqrt{k_n}}\sum_{i=1}^n \left(\Theta_i \epsilon_{\frac{R_i}{b(n/k_n)}}(t^{-1/\alpha},\infty)-\EE\bigl(\Theta_i \epsilon_{\frac{R_i}{b(n/k_n)}}(t^{-1/\alpha},\infty)\bigr)\right)\nonumber\\
&\Rightarrow W(t).
\end{align*}



\end{document}